%% file: faugerasm3as.tex
\documentstyle[twoside,psfig]{article}

\catcode`\@=11
\long\def\@makefntext#1{
\protect\noindent \hbox to 3.2pt {\hskip-.9pt  
$^{{\eightrm\@thefnmark}}$\hfil}#1\hfill}		%CAN BE USED 

\def\@makefnmark{\hbox to 0pt{$^{\@thefnmark}$\hss}}	%ORIGINAL 
	
\def\ps@myheadings{\let\@mkboth\@gobbletwo
\def\@oddhead{\hbox{}
\rightmark\hfil\eightrm\thepage}   
\def\@oddfoot{}\def\@evenhead{\eightrm\thepage\hfil
\leftmark\hbox{}}\def\@evenfoot{}
\def\sectionmark##1{}\def\subsectionmark##1{}}

\oddsidemargin=\evensidemargin
\newcounter{sectionc}\newcounter{subsectionc}\newcounter{subsubsectionc}
\renewcommand{\section}[1] {\vspace{12pt}\addtocounter{sectionc}{1} 
\setcounter{subsectionc}{0}\setcounter{subsubsectionc}{0}\noindent 
   {\tenbf\thesectionc. #1}\par\vspace{5pt}}
\renewcommand{\subsection}[1] {\vspace{12pt}\addtocounter{subsectionc}{1} 
	\setcounter{subsubsectionc}{0}\noindent 
   {\bf\thesectionc.\thesubsectionc. {\kern1pt \bfit #1}}\par\vspace{5pt}}
\renewcommand{\subsubsection}[1]{\vspace{12pt}\addtocounter{subsubsectionc}{1}
	\noindent{\tenrm\thesectionc.\thesubsectionc.\thesubsubsectionc.
	{\kern1pt \tenit #1}}\par\vspace{5pt}}

\newcounter{appendixc}
\newcounter{subappendixc}[appendixc]
\newcounter{subsubappendixc}[subappendixc]
\renewcommand{\thesubappendixc}{\Alph{appendixc}.\arabic{subappendixc}}
\renewcommand{\thesubsubappendixc}
	{\Alph{appendixc}.\arabic{subappendixc}.\arabic{subsubappendixc}}

\renewcommand{\appendix}[1] {\vspace{12pt}
        \refstepcounter{appendixc}
        \setcounter{figure}{0}
        \setcounter{table}{0}
        \setcounter{lemma}{0}
        \setcounter{theorem}{0}
        \setcounter{corollary}{0}
        \setcounter{definition}{0}
        \setcounter{equation}{0}
        \renewcommand{\thefigure}{\Alph{appendixc}.\arabic{figure}}
        \renewcommand{\thetable}{\Alph{appendixc}.\arabic{table}}
        \renewcommand{\theappendixc}{\Alph{appendixc}}
        \renewcommand{\thelemma}{\Alph{appendixc}.\arabic{lemma}}
        \renewcommand{\thetheorem}{\Alph{appendixc}.\arabic{theorem}}
        \renewcommand{\thedefinition}{\Alph{appendixc}.\arabic{definition}}
        \renewcommand{\thecorollary}{\Alph{appendixc}.\arabic{corollary}}
        \renewcommand{\theequation}{\Alph{appendixc}.\arabic{equation}}
        \noindent{\tenbf Appendix#1}\par\vspace{5pt}}
\newcommand{\subappendix}[1] {\vspace{12pt}
        \refstepcounter{subappendixc}
        \noindent{\bf Appendix \thesubappendixc. {\kern1pt \bfit #1}}
	\par\vspace{5pt}}
\newcommand{\subsubappendix}[1] {\vspace{12pt}
        \refstepcounter{subsubappendixc}
        \noindent{\rm Appendix \thesubsubappendixc. {\kern1pt \tenit #1}}
	\par\vspace{5pt}}

\newcommand{\textlineskip}{\baselineskip=13pt}
\newcommand{\smalllineskip}{\baselineskip=10pt}

\def\eightcirc{
\begin{picture}(0,0)
\put(4.4,1.8){\circle{6.5}}
\end{picture}}
\def\eightcopyright{\eightcirc\kern2.7pt\hbox{\eightrm c}} 

\newcommand{\copyrightheading}[1]
	{\vspace*{-2.5cm}\smalllineskip{\flushleft
	{\footnotesize Mathematical Models and Methods in 
                       Applied Sciences #1}\\
	{\footnotesize $\eightcopyright$\, World Scientific Publishing
	 Company}\\
	 }}

\def\abstracts#1#2#3{{
	\centering{\begin{minipage}{4.5in}\baselineskip=10pt\footnotesize
	\parindent=0pt #1\par 
	\parindent=15pt #2\par
	\parindent=15pt #3
	\end{minipage}}\par}} 

\def\keywords#1{{
	\centering{\begin{minipage}{4.5in}\baselineskip=10pt\footnotesize
	{\footnotesize\it Keywords}\/: #1
	 \end{minipage}}\par}}

\renewenvironment{thebibliography}[1]
	{\frenchspacing
	 \ninerm\baselineskip=11pt
	 \begin{list}{\arabic{enumi}.}
        {\usecounter{enumi}\setlength{\parsep}{0pt}     
	 \setlength{\leftmargin 17pt}{\rightmargin 0pt}  %FOR 10--99 ITEMS
         \setlength{\itemsep}{0pt} \settowidth
	{\labelwidth}{#1.}\sloppy}}{\end{list}}

\newcounter{itemlistc}
\newcounter{romanlistc}
\newcounter{alphlistc}
\newcounter{arabiclistc}

\newcommand{\fcaption}[1]{
        \refstepcounter{figure}
        \setbox\@tempboxa = \hbox{\footnotesize Fig.~\thefigure. #1}
        \ifdim \wd\@tempboxa > 5in
           {\begin{center}
        \parbox{5in}{\footnotesize\smalllineskip Fig.~\thefigure. #1}
            \end{center}}
        \else
             {\begin{center}
             {\footnotesize Fig.~\thefigure. #1}
              \end{center}}
        \fi}

\newcommand{\tcaption}[1]{
        \refstepcounter{table}
        \setbox\@tempboxa = \hbox{\footnotesize Table~\thetable. #1}
        \ifdim \wd\@tempboxa > 5in
           {\begin{center}
        \parbox{5in}{\footnotesize\smalllineskip Table~\thetable. #1}
            \end{center}}
        \else
             {\begin{center}
             {\footnotesize Table~\thetable. #1}
              \end{center}}
        \fi}

\def\@citex[#1]#2{\if@filesw\immediate\write\@auxout
	{\string\citation{#2}}\fi
\def\@citea{}\@cite{\@for\@citeb:=#2\do
	{\@citea\def\@citea{,}\@ifundefined
	{b@\@citeb}{{\bf ?}\@warning
	{Citation `\@citeb' on page \thepage \space undefined}}
	{\csname b@\@citeb\endcsname}}}{#1}}

\newif\if@cghi
\def\cite{\@cghitrue\@ifnextchar [{\@tempswatrue
	\@citex}{\@tempswafalse\@citex[]}}
\def\citelow{\@cghifalse\@ifnextchar [{\@tempswatrue
	\@citex}{\@tempswafalse\@citex[]}}
\def\@cite#1#2{{$\null^{#1}$\if@tempswa\typeout
	{IJCGA warning: optional citation argument 
	ignored: `#2'} \fi}}

\def\pmb#1{\setbox0=\hbox{#1}
	\kern-.025em\copy0\kern-\wd0
	\kern.05em\copy0\kern-\wd0
	\kern-.025em\raise.0433em\box0}

\def\fnt#1#2{\footnotetext{\kern-.3em
	{$^{\mbox{\scriptsize #1}}$}{#2}}}

\def\fpage#1{\begingroup
\voffset=.3in
\thispagestyle{empty}\begin{table}[b]\centerline{\footnotesize #1}
	\end{table}\endgroup}

\def\runninghead#1#2{\pagestyle{myheadings}
\markboth{{\protect\footnotesize\it{\quad #1}}\hfill}
{\hfill{\protect\footnotesize\it{#2\quad}}}}
\font\tenrm=cmr10
\font\tenit=cmti10 
\font\tenbf=cmbx10
\font\bfit=cmbxti10 at 10pt
\font\ninerm=cmr9

\font\eightrm=cmr8

\newtheorem{lemma}{Lemma}[sectionc]
\newtheorem{definition}{Definition}[sectionc]
\newtheorem{theorem}{Theorem}[sectionc]

\newtheorem{proposition}{Proposition}[sectionc]

\newenvironment{proof}
{\noindent {\bf Proof.}}{\qed \ . \\}
\def\fc{\hbox{\rm l\hskip -5pt 1}}
\def\Rr{\hbox{\rm I\hskip -1.5pt R}}
\def\ds{\displaystyle }

\def\ea{{\it et al. }}

\def\qed{\hbox{${\vcenter{\vbox{			%HOLLOW SQUARE
   \hrule height 0.4pt\hbox{\vrule width 0.4pt height 6pt
   \kern5pt\vrule width 0.4pt}\hrule height 0.4pt}}}$}}

\def\theequation{\thesectionc.\arabic{equation}}  %FOR SETTING EQ.~(1.1)

\begin{document}
\setlength{\textheight}{7.7truein}  %for 2nd page onwards

\runninghead{Well-posedness of a coupled biological-physical model for the upper ocean}
{Well-posedness of a coupled biological-physical model for the upper ocean}

\normalsize\textlineskip
\thispagestyle{empty}
\setcounter{page}{1}

\copyrightheading{}			%{Vol.~0, No.~0 (2001) 000--000}

\vspace*{0.88truein}

\fpage{1}
\centerline{\bf ON THE WELL-POSEDNESS OF A COUPLED} 
\baselineskip=13pt
\centerline{\bf ONE-DIMENSIONAL BIOLOGICAL-PHYSICAL MODEL} 
\baselineskip=13pt
\centerline{\bf FOR THE UPPER OCEAN}
\vspace*{0.37truein}
\centerline{\footnotesize BLAISE FAUGERAS}
\vspace*{0.015truein}
\centerline{\footnotesize\it MAPLY, Centre de math\'ematique INSA de Lyon,}
\centerline{\footnotesize\it Bat. L\'eonard de Vinci, 21 avenue Jean Capelle,} 
\centerline{\footnotesize\it 69621, Villeurbanne Cedex, France.}
\centerline{\footnotesize\it faugeras@laninsa.insa-lyon.fr}
\vspace*{10pt}
\vspace*{0.225truein}
%\pub{(Leave 1 inch blank space for publisher.)}

\vspace*{0.21truein}
\abstracts{This paper introduces a one-dimensional NPZD-model 
developed to simulate biological activity in a turbulent ocean water column. 
The model consists of a system of coupled semilinear parabolic equations. 
An initial-boundary value problem is formulated and the existence of a unique 
positive weak solution to it 
is proved. The existence result is derived using a variational formulation, 
an approximate model and a fixed-point method. It is shown that the qualitative analysis performed 
still applies if different parameterizations of several biological processes found in the
biogeochemical modeling literature are used.\\}{}{}
~\\
\keywords{Marine ecosystem model; semilinear parabolic system; variational formulation; positivity.\\
~\\
AMS Subject Classification: 86A05, 92F05, 35K45, 35K50, 35K57, 35R05}

\vspace*{1pt}\textlineskip	%) USE THIS MEASUREMENT WHEN THERE IS
\section{Introduction}
\vspace*{-0.5pt}
\noindent
Within the scope of global climate studies, authors carrying 
out modeling research 
use marine ecosystem models with increasing degrees of
complexity. Complexity in such models can arise from the number of
biological compartments, or state variables, which are taken into
account, as well as from the parameterizations used to model interactions
between these compartments. The number of variables 
can vary from one to more than ten. 
At least two variables, nutrients ($N$) and phytoplankton ($P$) are
necessary to model primary production, that is to say the
transformation of mineral nutrients into primitive biotic material
using external energy, provided by the sun (Taylor \ea
\cite{Taylor:1991}). 
However, in order to study the ocean carbon
cycle, the main biological processes which have to be understood and
estimated are primary production, but also the export of 
organic matter from the surface to deep ocean layers and organic matter
remineralization. The simpliest model able to represent all these
processes contains four variables, nutrients ($N$), phytoplankton
($P$), zooplankton ($Z$) and detritus ($D$). This type of model is
termed the NPZD-model and different variants of it are used in
numerous studies. All these models are similar from a structural 
point of view but authors use different parameterizations to model 
fluxes between biological compartments.

Complexity can also arise from the spatial resolution of the
physical dynamics to which biological variables are submitted. Many
model set-ups are zero-dimensional and biological variables correspond
to ocean mixed-layer values (e.g., Fasham \ea \cite{Fasham:1990}, Steele and
Henderson \cite{Steele:1992}, Spitz \ea \cite{Spitz:1998}, 
Fennel \ea \cite{Fennel:2001}). Others are one-dimensional, considering
that the ocean, in some particular places, can be modeled with a good
approximation by a turbulent water column 
(e.g., Prunet \ea \cite{Prunet:1996a}, Doney \ea
\cite{Doney:1996}, L\'evy \ea \cite{Levy:1998a}, M\'emery \ea
\cite{Memery:2002}). Finally, in some studies, the biological model
is integrated in a three-dimensional circulation model (e.g., Fasham \ea
\cite{Fasham:1993}, Moisan \ea \cite{Moisan:1996}, 
L\'evy \ea \cite{Levy:1998b}
, Carmillet \ea \cite{Carmillet:2001}). The question which has motivated
this work is: are all these models well-posed? Of course, it seems
difficult to study all of them and in this work we concentrate first 
on a one-dimensional NPZD-model and then discuss the possible
generalization of our result. The
three-dimensional version of the model we consider is proposed 
in L\'evy \ea \cite{Levy:1998b}, and a one-dimensional version 
of a similar model, containing six biological variables, 
is used by Faugeras \ea
\cite{Faugeras:2002} 
to assimilate data from the JGOFS-DYFAMED time-series 
station in the North-Western Mediterranean Sea. 

Mathematically, the
biological model under consideration is a system of 
coupled parabolic semilinear equations to which initial and boundary 
conditions are added. Under certain hypotheses, 
this general type of initial-boundary value problem can be transformed to 
an abstract Cauchy problem and studied using the theory of semigroups 
(following Chapter 6 of the book by Pazy \cite{Pazy:1983} for example). 
In their paper Boushaba \ea \cite{Boushaba:2002} used results on
semigroups to provide a mathematical analysis of a model describing
the evolution of a single variable phytoplankton. Although 
the model they considered is three-dimensional the biological 
reaction terms are quite simple since only production 
and mortality of phytoplanckton are represented. The model 
we propose here seems to be more realistic 
and has already been numerically 
validated using observations from the DYFAMED time-series station 
(L\'evy \ea \cite{Levy:1998b}, Faugeras \ea \cite{Faugeras:2002}).

If data are regular enough the semigroup method can enable 
the existence of classical solutions to be proved. 
However it does not enable parabolic equations with time-dependent 
irregular coefficients to be easily handled. Since this is the case in the 
NPZD-model we consider a variational formulation approach 
is more attractive.
The main purpose of this paper is to address the issue of the existence 
of weak solutions to this particular one-dimensional model. 
The method we propose is inspired from the
work of Artola \cite{Artola:1989} in which an existence result 
for a semilinear parabolic system is derived using a fixed-point 
argument. We introduce an approximate model and 
prove the existence of weak solutions to this model using this method.
We then pass to the limit in the approximate model to prove the existence 
of weak solutions to the NPZD-model. 
Furthermore, as the variables of the model represent concentrations 
they should be positive. We show this is the case. 

We shall now briefly outline the contents of the paper. In the next
section we introduce the equations of the one-dimensional NPZD-model
and give some comments on the different parameterizations used. 
In Section 3 we set
the mathematical framework and state our main result, which is proved
in Sections 4 and 5. The goal of Section 6 is twofold. First we show
that the existence and positivity results still hold when different
parameterizations found in the literature are used. 
Second, we address the issue of uniqueness
of solutions. In order to prove uniqueness we need the 
nonlinear reaction terms 
to satisfy a local Lipschitz condition. We show this is the 
case in our particular model.  

\vspace*{1pt}\textlineskip	%) USE THIS MEASUREMENT WHEN THERE IS
\section{Presentation of the one-dimensional NPZD-model}
\label{section:NPZD}
\vspace*{-0.5pt}
\noindent 
\subsection{Equations of the model}
\label{subsec:NPZD}
\noindent
In this section we give the equations of the one-dimensional NPZD-model and formulate 
the initial-boundary value problem which will be studied.

Let us first of all justify the use of a one-dimensional model. We have in mind numerical studies 
(Faugeras \ea \cite{Faugeras:2002}, L\'evy \ea \cite{Levy:1998a}, M\'emery \ea \cite{Memery:2002}) 
conducted with such one-dimensional models. In these studies simulations are forced with 
physical data (wind stress, heat fluxes, evaporation-precipitation) 
and validated by comparison with biogeochemical data (chlorophyll and nitrate) collected at the DYFAMED station. 
This station is located in the Northwestern Mediterranean Sea and is an interesting test case for several reasons. 
First, several biogeochemical production regimes that take place in the world ocean are found here. Secondly, the 
station is far enough away from the Ligurian Current to be sufficiently protected from lateral transport, 
thereby permitting a one-dimensional study.\\
In the above cited numerical studies the biogeochemical model is integrated 
in a one-dimensional physical model, which simulates the time evolution of velocity, temperature,
salinity and turbulent kinetic energy (TKE). Advection is neglected even though this might result 
in a crude approximation in summer during strong wind events (Andersen and Prieur \cite{Andersen:2000}). 
The only dynamic process which is taken into account is vertical diffusion.

The one-dimensional NPZD-model consists of four coupled semilinear parabolic equations.
Before introducing them let us give some notations. In all the following 
we denote the nutrient, phytoplankton, zooplankton and detritus
concentration vector by,
$$
{\bf C}=(N,P,Z,D)=(C_1,C_2,C_3,C_4),
$$
and the reaction terms by, 
$$
{\bf f}=(f_N,f_P,f_Z,f_D)=(f_{1},f_2,f_3,f_4).
$$ 
The equations of the NPZD-model read as follows. For $i=1$ to $4$:

\begin{equation}
\label{eqn:p1}
\left \lbrace \begin{array}{ll}
\ds \frac{\partial C_i}{\partial t}- \ds \frac{\partial}{\partial x}(d(t,x)
 \ds \frac{\partial C_i}{\partial x}) + \delta_{i,4} v_d \frac{\partial
C_i}{\partial x} = f_i(t,x,{\bf C}), & t \in ]0,T], \quad x \in ]0,L[, \\[7pt]
\ds \frac{\partial C_i}{\partial x}(t,0)=\ds \frac{\partial C_i}{\partial
x}(t,L)=0, & t \in ]0,T],\\[7pt]
C_i(0,x)=C_i^0(x), & x \in ]0,L[, \\
\end{array}
\right.
\end{equation}
with
\begin{equation}
\label{eqn:p2}
\left \lbrace \begin{array}{lll}
f_N(t,x,{\bf C})&=&(-\mu_p(1-\gamma) L_I(t,x,P) L_{N}P+\mu_z Z + \mu_d D)\ \fc_{]0,l]}(x)\\[7pt] 
          & &+(\tau(P+Z+D))\ \fc_{]l,L[}(x),\\[7pt]
f_P(t,x,{\bf C})&=&(\mu_p(1-\gamma)L_I(t,x,P)L_NP-G_PZ-m_pP)\ \fc_{]0,l]}(x)\\[7pt]
          & &+(-\tau P)\ \fc_{]l,L[}(x),\\[7pt]
f_Z(t,x,{\bf C})&=&(a_pG_PZ + a_dG_DZ -m_zZ -\mu_zZ)\ \fc_{]0,l]}(x)\\[7pt]
          & &+(-\tau Z)\ \fc_{]l,L[}(x),\\[7pt]
f_D(t,x,{\bf C})&=&((1-a_p)G_PZ-a_dG_DZ+m_pP+m_zZ-\mu_dD)\ \fc_{]0,l]}(x)\\[7pt]
          & &+(-\tau D)\ \fc_{]l,L[}(x).\\
\end{array}
\right.
\end{equation}
\noindent
$T$ is a fixed time. In numerical simulations, system (\ref{eqn:p1}) is intregated over a period of time which can 
vary from one month to a few years.\\
~\\
$L$ is the depth of the water column under consideration ($L \approx 1000$ m), $l$ is the
maximum depth of the euphotic layer ($l \approx 200$ m).\\
~\\
$\fc_{]0,l]}$ and $\fc_{]l,L[}$ are the usual indicator functions,
$$
\fc_{]0,l]}(x)=
\left \lbrace \begin{array}{l}
1 \quad {\mathrm{if}} \quad x \in ]0,l],\\
0 \quad {\mathrm{otherwise}}.\\
\end{array}
\right.
$$
\noindent
$\delta_{i,4}$ is the Kronecker symbol, 
$$
\delta_{i,4} =
\left \lbrace \begin{array}{l}
1 \quad {\mathrm{if}} \quad i=4,\\
0 \quad {\mathrm{otherwise}}.
\end{array}
\right.
$$
\noindent
~\\
Neuman boundary conditions at $x=0$ and $x=L$ express the fact that there is no flux through the surface of the ocean 
and through the ocean floor.\\
~\\
Initial concentrations, $C_i^0$, satisfy $C_i^0(x) \ge 0$ for all $x \in ]0,L[$.\\
~\\
The different parameters which appear in the reaction terms $f_i$ are strictly positive
constants. All of them are shown in Table \ref{tab:paramNPZD}. 
A schematic representation of
the model is shown on Figure \ref{fig:NPZD}. 
Let us note that parameters 
$\gamma, a_p$ and $a_d$ satisfy $1-\gamma > 0$, $1-a_p> 0$ and $1-a_d > 0$.\\ 
The nonlinear functions $L_I, L_N, G_P$ and $G_D$ are given explicitly 
in the following subsection, and more details about the model can be found 
in L\'evy \ea \cite{Levy:1998b}. 

\begin{table}[htbp]
\caption{ Parameter values \label{tab:paramNPZD} }
\begin{tabular}{|l|c|c|c|}
\hline
parameter& name & value & unit \\
\hline
half-saturation\ constant 
& $k_n$ & 0.5 & $mmolNm^{-3}$ \\ 

maximal\ grazing\ rate& $g_z$ &0.75 &$day^{-1}$\\
 
half-saturation\ constant\ for\ grazing& $k_z$ & 1 & $mmolN.m^{-3}$\\

assimilated\ fraction\ of\ phytoplankton 
& $a_p$ & 0.7 & \\

assimilated\ fraction\ of\ detritus 
& $a_d$ & 0.5 & \\

zooplankton\ excretion\ rate 
&$\mu_z$&0.1&$day^{-1}$\\

phytoplankton\ mortality\ rate
&$m_p$ &0.03 & $day^{-1}$\\

zooplankton\ mortality\ rate
& $m_z$ &0.03 &$day^{-1}$\\

detritus\ remineralization\ rate& $\mu_d$ & 0.09 & $day^{-1}$\\

detritus\ sedimentation\ speed &$v_d$& 5 & $m.day^{-1}$\\

maximal\ growth\ rate\ & $\mu_p$ & 2 & $day^{-1}$\\

exsudation\ fraction & $\gamma$ & 0.05 & \\

remineralization\ rate&$\tau$ & 0.05 &$day^{-1}$\\

\hline
\end{tabular}
\end{table}

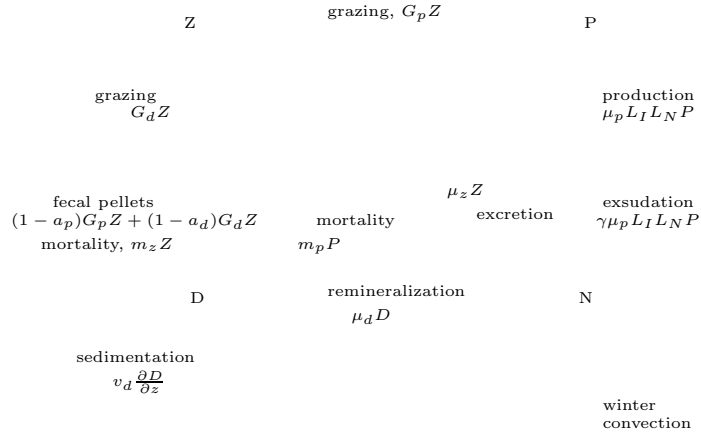
\begin{figure*}[!h]
\caption{\label{fig:NPZD}
Schematic representation of the compartments and processes of the NPZD
surface layer model.}
\centering
\input{schemaNPZD_english.tex}
\end{figure*}

\subsection{Comments and hypotheses}
\label{ssec:comethyp}
\noindent
\begin{enumerate}
\item 
The mixing or diffusion coefficient, $d(t,x)$, is obtained diagnostically from TKE 
(Gaspar \ea \cite{Gaspar:1990}). 
In modeling studies it is considered in the first approximation that 
biological variables do not influence physical variables. As a consequence biological tracers 
are vertically mixed with the same coefficient as temperature and salinity. 
This coefficient is an output of 
the physical model and data for the biological model. 
Consequently it does not depend on ${\bf C}$. 
We are thus given once and for all a mixing 
coefficient $d(t,x)$. It strongly varies in space and time and we can not assume it is 
particularly regular (Lewandosky \cite{Lewandosky:1997}). The usual basic 
assumption made in the mathematical literature as well as in numerical studies is the following. 
We suppose that 
$$
0 < d_0 \le d(t,x) \le d_{\infty},\ a.e.\ {\mathrm{in}}\ ]0,T[ \times ]0,L[.
$$ 

\item  
Because of functions $\fc_{]0,l]}$ and $\fc_{]l,L[}$, 
the equations of the model are not the same above and below
the depth $l$ which physically corresponds to the depth at which the
action of light on the system becomes negligeable. This
corresponds to a discontinuity of the reaction terms $f_i(t,x,{\bf C})$ at
the point $x=l$. It is the choice of modelization made by L\'evy \ea \cite{Levy:1998b}. 
Above the depth $l$ the reaction terms correspond to the schematic representation of
the model shown on Figure \ref{fig:NPZD}. The basic biogeochemical fluxes are represented using 
a minimum number of prognostic variables. Nutrients allow the estimation 
of production to be made. Zooplankton 
mortality and detrital sedimentation feed the particle export flux. Below the depth $l$ 
remineralization processes are preponderent and the surface model does not apply. 
Instead decay of phytoplankton, zooplankton and detritus in nutrients parameterize remineralization. 
More details about the modeled biogeochemical processes can be found in 
L\'evy \ea \cite{Levy:1998b}. In the following points we give the analytical expression of the nonlinear 
terms which are used.

\item  
$L_N$, $G_P$ and $G_D$ are nonlinear functions.
  \begin{itemize}
  \item $L_N$ parameterizes the nutrient limitation on phytoplankton
growth. It follows the Michaelis-Menten kinetic, 
$L_{N}=\ds \frac{N}{k_n+N}$. The possible nullification of the term
$k_n+N$, invites us to define, $L_{N}=\ds \frac{N}{k_n+|N|}$. 
This formulation will be used in the following. 
We will show that if initial concentrations are positive
then concentrations always stay positive, thus the two formulations
are equivalent. Let us remark that,
\begin{itemize}
  \item[] $L_N$ is defined and continuous on $\Rr$,
  \item[] $|L_N(N)| \le 1$, $\forall N \in \Rr$. 
\end{itemize}
  \item $G_P$ and $G_D$ are the zooplankton grazing rates on
phytoplankton and detritus. The formulation used is a 
squared Michaelis-Menten response function:
$$
\begin{array}{l}
G_P=\ds \frac{g_zP^2}{k_z+P^2},\\
G_D=\ds \frac{g_zD^2}{k_z+D^2}.\\
\end{array}
$$ 
In the remainder of this paper we use the following properties:  
\begin{itemize}
  \item[] $G_P$ and $G_D$ are defined and continuous on $\Rr$, 
  \item[] $|G_P| \le g_z $, $\forall P \in \Rr$, 
  \item[] $|G_D| \le g_z $, $\forall D \in \Rr$. 
\end{itemize}
\end{itemize}

\item 
The limitation of phytoplankton growth by light is parameterized by,
$$
L_I(t,x,P)=1-\exp(-PAR(t,x,P)/k_{par}),
$$ 
$k_{par}$ is a positive constant. The photosynthetic available radiation,
$PAR$, is predicted from surface
irradiance and phytoplankton pigment content by a light
absorption model according to L\'evy \ea \cite{Levy:1998b}. From a biological
point of view, the fact that $PAR$ depends on $P$ is important. This
models the so-called self-shading effect. We give further details of the
parameterization of $PAR$ in Section 6, and here we
only suppose it is a positive function, continuous in $P$ for
a.e. $t,x$ and measurable 
in $t,x$ for all $P$. In order to
prove the existence result we have to notice that:
\begin{itemize}
  \item[] $L_I$ is defined on $[0,T] \times [0,L] \times \Rr$,
  \item[] $0 \le L_I(t,x,P) \le 1$, a.e in $[0,T] \times [0,L] \times \Rr$, 
  \item[] $(t,x) \rightarrow L_I(t,x,P)$ is measurable, for all $P \in
  \Rr$,
  \item[] $P \rightarrow L_I(t,x,P)$ is continuous, for a.e $(t,x) \in
  [0,T] \times [0,L]$.
\end{itemize}

\item
Eventually, let us remark the presence of the advection term 
$v_d \ds \frac{\partial D}{\partial x}$ in the detritus equations. 
Detritus, $D$, sink at a speed of $v_d$. 
\end{enumerate}

\vspace*{1pt}\textlineskip	%) USE THIS MEASUREMENT WHEN THERE IS
\section{Mathematical preliminaries and statement of main result}
\vspace*{-0.5pt}
\noindent 
\subsection{Functional spaces}
\noindent
In this section we introduce the functional spaces which we use in the
remainder of this work. All this study is conducted on the open set $]0,L[$ and 
$T$ is a fixed time. 
Throughout this work, concentrations, $C_i$, are considered as
elements of the functional space $L^2(0,L)$ whose Hilbert space 
structure is convenient
to use. However, let us remember that $L^2(0,L)$ is continuously imbedded into $L^1(0,L)$
which is a natural space for concentrations.

${\bf H}$ and ${\bf H}^1$ are the separable Hilbert spaces defined by
$$
\begin{array}{l}
{\bf H}=(L^2(0,L))^4,\\
{\bf H}^1=(H^1(0,L))^4.
\end{array}
$$
${\bf H}$ is equipped with the scalar product
$$
\begin{array}{ll}
({\bf C},\hat{{\bf C}})&= 
 \ds \int_{0}^{L} \ds \sum_{i=1}^{4} C_i(x)\hat{C}_i(x)dx\\[7pt]
&=  \ds \sum_{i=1}^{4} (C_i,\hat{C}_i)_{L^2(0,L)}.\\
\end{array}
$$
We denote by $||.||$ the induced norm on ${\bf H}$.\\

${\bf H}^1$ is equipped with the scalar product
$$
\begin{array}{ll}
({\bf C},\hat{{\bf C}})_1
&= \displaystyle \int_{0}^{L} \ds \sum_{i=1}^{4} C_i(x)\hat{C}_i(x)dx+ \displaystyle \int_{0}^{L}
 \ds \ds \sum_{i=1}^{4} \ds \frac{\partial C_i(x)}{\partial x} \ds \frac{\partial
\hat{C}_i(x)}{\partial x}dx\\[7pt]
&=  \ds \ds \sum_{i=1}^{4} (C_i,\hat{C}_i)_{L^2(0,L)}+
 \ds \ds \sum_{i=1}^{4} (\ds \frac{\partial C_i}{\partial x},\ds \frac{\partial
\hat{C}_i}{\partial x})_{L^2(0,L)}.\\
\end{array}
$$
We denote by $||.||_1$, the induced norm on ${\bf H}^1$.

We will also have to consider the space ${\bf L}^{\infty} = (L^{\infty}(0,L))^4$. 
$L^{\infty}(0,L)$ is a Banach equipped with the norm 
$$
||C_i||_{\infty} = inf \lbrace M;|C_i(x)| \le M \ a.e. \ in \ (0,L) \rbrace.
$$ 
Similarly ${\bf L}^{\infty}$ is a Banach space equipped with the norm
$$
||{\bf C}||_{\infty} = 
\sup_{i=1, ..., 4}
||C_i||_{\infty}.
$$

Now, if $X$ is a real Banach space equipped with the norm $||.||_X$,
$C([0,T],X)$ is the space of continuous functions on $[0,T]$ with
values in $X$, equipped with the norm,
$$
||{\bf C}||_{C([0,T],X)}=\sup_{[0,T]}
||{\bf C}(t)||_X.
$$
Similarly $L^2(0,T,X)$ is the space of functions $L^2$ in
time with values in $X$, equipped with the norm,
$$
||{\bf C}||_{L^2(0,T,X)}=(\int_{0}^{T}||{\bf C}(t)||_X^2 dt)^{1/2},
$$
and $L^{\infty}(0,T,X)$ is the space of functions $L^{\infty}$ in time
with values in $X$, equipped with the norm, 
$$
||{\bf C}||_{L^{\infty}(0,T,X)}= inf \lbrace M;||{\bf C}(t)||_X 
\le M 
\ a.e \ in \ (0,T) \rbrace.
$$
$C([0,T],X)$, $L^2(0,T,X)$ and $L^{\infty}(0,T,X)$ are Banach spaces.

We have the useful

\begin{lemma}
\label{lemma:injection}
The imbedding, ${\bf H}^1 \subset {\bf L}^{\infty}$, is continuous.\\
The imbeddings, ${\bf H}^1 \subset {\bf H}$ and ${\bf H}^1 \subset
C([0,L],\Rr^4)$ are compact.\\
\end{lemma}
\begin{proof}
It is a consequence of corollaries IX.14 and IX.16 in
Br\'ezis \cite{Brezis:1992}, 
and of the Rellich-Kondrachoff theorem (Lions and
Magenes \cite{Lions:1968})
\end{proof}

${\bf H}'$ denotes the dual of ${\bf H}$ and $({\bf H}^1)'$ the
dual of ${\bf H}^1$. When ${\bf H}$ is identified with its dual, we
have the classical scheme, 
$$
{\bf H}^1 \subset {\bf H}={\bf H}' \subset ({\bf H}^1)',
$$
where each space is dense in the following and the imbeddings are continuous.

Let us denote by $W({\bf H}^1)$ the Hilbert space,
$$
W({\bf H}^1) = \lbrace {\bf C} \in
L^2(0,T,{\bf H}^1); \frac{d {\bf C}}{dt} \in L^2(0,T,({\bf H}^1)')
\rbrace.
$$ 

\begin{lemma}
Every ${\bf C} \in W({\bf H}^1)$ is a.e equal to a continuous
function from $[0,T]$ to ${\bf H}$. Moreover we have the following
continuous imbedding,
$$
W({\bf H}^1) \subset C([0,T],{\bf H}).
$$
\end{lemma}
\begin{proof}
See Dautray and Lions \cite{Dautray:1988b} for example
\end{proof}

Moreover, because the injective mapping ${\bf H}^1 \subset {\bf H}$
is compact, we know that,

\begin{lemma}
\label{lemma:compacite}
The identity mapping, $W({\bf H}^1) \subset L^2(0,T,{\bf H})$, is compact.\\
\end{lemma}
\begin{proof}
See Aubin \cite{Aubin:1963} or Lions \cite{Lions:1969}
\end{proof}

\subsection{A preliminary transformation of the system and the bilinear form
$a(t,{\bf C},{\bf C}')$}
\noindent
In order to work with a bilinear form as simple as possible, we start
by adding $\lambda C_i$ to both sides of system (\ref{eqn:p1}). The
value of $\lambda > 0$ will be fixed in what follows. This leads to
the equivalent system, for $i=1$ to $4$:
\begin{equation}
\label{eqn:p1equi}
\left \lbrace \begin{array}{ll}
\ds \frac{\partial C_i}{\partial t}- \ds \frac{\partial}{\partial x}(d(t,x)
 \ds \frac{\partial C_i}{\partial x}) + \delta_{i,4} v_d \frac{\partial
C_i}{\partial x} + \lambda C_i&\\[7pt]
= f_i(t,x,{\bf C}) + \lambda C_i,& t \in ]0,T],\ x \in ]0,L[, \\[7pt]
\ds \frac{\partial C_i}{\partial x}(t,0)=\ds \frac{\partial C_i}{\partial
x}(t,L)=0, & t \in ]0,T],\\[7pt]
C_i(0,x)=C^0(x), & x \in ]0,L[. \\
\end{array}
\right.
\end{equation}

For $N,N',P,P',Z,Z'$ and $D,D' \in H^1(0,L)$, we define
$$
\begin{array}{l}
a_N(t,N,N')= \ds \int_{0}^{L} d(t,x) \ds \frac{\partial N}{\partial
x} \ds \frac{\partial N'}{\partial x} + \lambda \int_{0}^{L} NN',\\[7pt]
a_P(t,P,P')= \ds \int_{0}^{L}  d(t,x) \ds \frac{\partial P}{\partial
x} \ds \frac{\partial P'}{\partial x}+\lambda \int_{0}^{L} PP',\\[7pt]
a_Z(t,Z,Z')= \ds \int_{0}^{L} d(t,x) \ds \frac{\partial Z}{\partial
x} \ds \frac{\partial Z'}{\partial x}+\lambda \int_{0}^{L} ZZ',\\[7pt]
a_D(t,D,D')= \ds \int_{0}^{L} d(t,x) \ds \frac{\partial D}{\partial
x} \ds \frac{\partial D'}{\partial x}+\int_{0}^{L} v_d \ds \frac{\partial D}{\partial
x}D' +\lambda \int_{0}^{L} DD',\\
\end{array}
$$ 
and
$$
a(t,{\bf C},{\bf C}')=a_N(t,N,N')+a_P(t,P,P')+a_Z(t,Z,Z')+a_D(t,D,D').
$$

\begin{lemma}
\label{lemma:formebi}
For a.e. $t \in [0,T]$, $a(t,{\bf C},{\bf C}')$ is a continuous
bilinear form on ${\bf H}^1 \times {\bf H}^1$. For all ${\bf C}$,
${\bf C}' \in {\bf H}^1$, $t \rightarrow a(t,{\bf C},{\bf C}')$ is
measurable and there exists a constant $M_a > 0$ such that,
$$
|a(t,{\bf C},{\bf C}')| \le M_a ||{\bf C}||_1
 ||{\bf C}'||_1, \quad \forall {\bf C},{\bf C}' \in {\bf H}^1.
$$
For a fixed $\lambda$, $\lambda \ge \ds \frac{v_d^2}{2d_0}$, there
exists a constant $c_0 > 0$ such that,
$$
a(t,{\bf C},{\bf C}) \ge c_0 ||{\bf C}||_1^2, \quad \forall
t\in[0,T], \quad \forall {\bf C} \in {\bf H}^1.
$$
\end{lemma}
\begin{proof}
The proof for this is classical and it is omited
\end{proof}

\subsection{The reaction terms and the nonlinear operator ${\bf G}$}
\label{section:G}
\noindent
In this paragraph we show that the reaction terms of the NPZD-model
enable us to define a continuous operator ${\bf G}$ on $L^2(0,T,{\bf H})$.

\begin{lemma}
\label{lemma:besogne}
The reaction terms $f_N$, $f_P$, $f_Z$ and $f_D$ defined in Section 2 
have the following properties:
\begin{itemize}
\item[(P1)] For a.e. $(t,x) \in [0,T] \times [0,L]$, and all ${\bf C} \in \Rr^4$,
$$
\begin{array}{l}
|f_N(t,x,{\bf C})| \le (\mu_p(1-\gamma) + \tau)|P| +(\mu_z + \tau)|Z|+(\mu_d + \tau)|D|,\\[4pt]
|f_P(t,x,{\bf C})| \le (\mu_p(1-\gamma) +m_p+ \tau)|P| +g_z|Z|,\\[4pt]
|f_Z(t,x,{\bf C})| \le ((a_p + a_d)g_z +m_z +\mu_z+ \tau)|Z|,\\[4pt]
|f_D(t,x,{\bf C})| \le (((1-a_p)+a_d)g_z+m_z)|Z| +m_p|P| +(\mu_d + \tau)|D|.\\  
 \end{array}
$$
\item[(P2)] The  function ${\bf f}(t,x,{\bf C})$, defined from $[0,T] \times [0,L]
\times \Rr^4 \rightarrow \Rr^4$, is measurable in $(t,x)$, for all
${\bf C} \in \Rr^4$, and is continuous in ${\bf C}$, for a.e. $(t,x) \in
[0,T] \times [0,L]$.
\end{itemize}
\end{lemma}
\begin{proof}
The proof is straightforward and uses the comments of
Section 2.2
\end{proof}

We now define a function ${\bf g}(t,x,{\bf C}) =
{\bf f}(t,x,{\bf C}) + \lambda {\bf C}$ from $[0,T] \times [0,L]
\times \Rr^4 \rightarrow \Rr^4$ and a nonlinear operator, ${\bf G}$, by:
$$
{\bf GC}={\bf g}(t,x,{\bf C}(t,x)), \quad (t,x) \in [0,T] \times [0,L].
$$

\begin{proposition}
\label{prop:biendefG}
The operator, ${\bf G}$, is well defined from $L^2(0,T,{\bf H})$ to
itself. There exists a constant $M_g > 0$, depending only on the
parameters of the model, such that, for all ${\bf C} \in \L^2(0,T,{\bf H})$
$$
||{\bf GC}||_{L^2(0,T,{\bf H})} \le M_g ||{\bf C}||_{L^2(0,T,{\bf H})}.  
$$
The operator ${\bf G}$ is continuous on $L^2(0,T,{\bf H})$.\\
\end{proposition}
\begin{proof}
Let ${\bf C} \in L^2(0,T,{\bf H})$ 
and $t \in [0,T]$. From point $(P1)$ of lemma \ref{lemma:besogne} we obtain,
$$
\begin{array}{l}
||{\bf G}{\bf C}(t)||^2=\\[7pt] 
\ds \int_{0}^{L} |f_N(t,x,{\bf
C}(t,x)) + \lambda N(t,x)|^2 
+ |f_P(t,x,{\bf C}(t,x))+\lambda P(t,x)|^2\\[7pt]  
+ |f_Z(t,x,{\bf C}(t,x))+\lambda Z(t,x)|^2 + 
|f_D(t,x,{\bf C}(t,x))+\lambda D(t,x)|^2 dx,\\[7pt] 
 ||{\bf G}{\bf C}(t)||^2 \le \\[7pt]   
cte_1 ( ||P(t)||_{L^2(0,L)}^2 + ||Z(t)||_{L^2(0,L)}^2 +
||D(t)||_{L^2(0,L)}^2+ ||N(t)||_{L^2(0,L)}^2 )\\[7pt] 
 + cte_2 (||P(t)||_{L^2(0,L)}^2+||Z(t)||_{L^2(0,L)}^2+ ||P(t)||_{L^2(0,L)}^2 )\\[7pt] 
 +cte_3 (||Z(t)||_{L^2(0,L)}^2)\\[7pt] 
 + cte_4 (||Z(t)||_{L^2(0,L)}^2+||P(t)||_{L^2(0,L)}^2+||D(t)||_{L^2(0,L)}^2),\\[7pt] 
||{\bf G}{\bf C}(t)||^2 \le M_g^2 ||{\bf C}(t)||^2, 
\end{array}
$$
and integrating on $[0,T]$,
$$
||{\bf G}{\bf C}||_{L^2(0,T,{\bf H})}  \le M_g ||{\bf C}||_{L^2(0,T,{\bf H})}.
$$
From point ($P2$) of lemma \ref{lemma:besogne} we know that the function,
$$
{\bf g}(t,x,{\bf C}) = {\bf f}(t,x,{\bf C}) + \lambda {\bf C},
$$ 
from $[0,T] \times [0,L] \times \Rr^4 \rightarrow \Rr^4$, satisfies
the conditions of Carath\'eodory  and by theorem 2.1 page 22 of
Krasnosel'skii 
\cite{Krasnoselskii:1964}, we know that the operator ${\bf G}$ is continuous
\end{proof}

\subsection{Variational formulation}
We can now write the definition of a weak solution to system (\ref{eqn:p1}),
\begin{definition}
${\bf C} \in W({\bf H}^1)$ is a weak solution of system (\ref{eqn:p1}) if 
$$
\forall \phi \in {\bf H}^1, \quad (\ds \frac{d {\bf C}}{dt}, {\bf \phi}) +
a(t,{\bf C},{\bf \phi}) = ({\bf GC},{\bf \phi}),
$$ 
in the ${\mathcal{D}}'(]0,T[)$ sens,\\ 
and ${\bf C}(0)={\bf C}^0$. 
\end{definition}
and state the main result of this paper, 
\begin{theorem}
\label{theo:main}
Let ${\bf C}^0 \in {\bf H}$. There exists a weak solution to system
(\ref{eqn:p1}). 
Furthermore, if $N^0,P^0,Z^0$ and $D^0$ are positive then $N,P,Z$ and $D$ are
positive for a.e. $t \in [0,T]$. 
\end{theorem}
The proof is given in the next two sections.

\vspace*{1pt}\textlineskip	%) USE THIS MEASUREMENT WHEN THERE IS
\section{Existence}
\vspace*{-0.5pt}
\noindent
The existence result is obtained in two steps. We first define an
approximate problem, in which the operator ${\bf G}$ is approximated 
by an operator ${\bf G}_n$. This approximate problem is solved
using the Schauder fixed-point theorem. In the second step we let $n \rightarrow
\infty$ to obtain a solution to the initial problem.  
 
\subsection{Step 1: approximated problem}
\noindent
Let $n > 0$ be a fixed integer and ${\bf g}_n$ be defined by,
$$
\begin{array}{lll}
{\bf g}_n: & [0,T] \times [0,L] \times \Rr^4 & \rightarrow \Rr^4\\
& (t,x,{\bf C}) & \rightarrow \ds (\frac{({\bf g}(t,x,{\bf C}))_i}
{1+ \frac{1}{n} |({\bf g}(t,x,{\bf C}))_i|})_{i=1, ...4}.
\end{array}
$$
Define the nonlinear operator, ${\bf G}_n$, by:
$$
{\bf G}_n{\bf C}={\bf g}_n(t,x,{\bf C}(t,x)), \quad (t,x) \in [0,T] \times [0,L].
$$

\begin{proposition}
\label{prop:biendef2}
The operator, ${\bf G}_n$, is well defined from $L^2(0,T,{\bf H})$
to itself and there exists a constant, $M_g > 0$, such that 
for all ${\bf C} \in \L^2(0,T,{\bf H})$,
$$
||{\bf G}_n{\bf C}||_{L^2(0,T,{\bf H})} \le M_g ||{\bf C}||_{L^2(0,T,{\bf H})}.  
$$
The operator ${\bf G}_n$ is continuous on $L^2(0,T,{\bf H})$.\\
For all ${\bf C} \in L^2(0,T,{\bf H})$, we also have the estimation,
$$
||{\bf G}_n{\bf C}||_{L^2(0,T,{\bf H})} 
\le 2n\sqrt{LT}.
$$
\end{proposition}
\begin{proof}
Let ${\bf C} \in L^2(0,T,{\bf H})$. From the definition of
${\bf g}_n$ 
and from proposition \ref{prop:biendefG} we obtain,
$$
||{\bf G}_n{\bf C}||_{L^2(0,T,{\bf H})} 
\le ||{\bf G}{\bf C}||_{L^2(0,T,{\bf H})} \le M_g ||{\bf C}||_{L^2(0,T,{\bf H})}. 
$$ 
The estimation $||{\bf G}_n{\bf C}||_{L^2(0,T,{\bf H})} 
\le 2n\sqrt{LT}$ is also derived easily from the choice made to define
${\bf g}_n$.\\
As in proposition \ref{prop:biendefG}, ${\bf g}_n$
satisfies the conditions of Carath\'eodory and ${\bf G}_n$ is continuous
on $L^2(0,T,{\bf H})$
\end{proof}

We now seek a solution to the approximated system and show that such a
solution is a fixed-point of the operator $\Theta$ defined in the next
proposition.

\begin{proposition}
\label{prop:theta}
Let $\hat{{\bf C}}$ be a fixed element of $L^2(0,T,{\bf H})$ and let
${\bf C}^0 \in {\bf H}$. There exists a unique solution to the problem:\\
find ${\bf C} \in W({\bf H}^1)$ such that, 
$$
\forall \phi \in {\bf H}^1, \quad (\ds \frac{d {\bf C}}{dt}, {\bf \phi}) +
a(t,{\bf C},{\bf \phi}) = ({\bf G}_n\hat{{\bf C}},{\bf \phi}),
$$ 
in the ${\mathcal{D}}'(]0,T[)$ sens,\\ 
and ${\bf C}(0)={\bf C}^0$.\\
This solution defines an operator $\Theta$ on
$L^2(0,T,{\bf H})$, 
$\Theta\hat{{\bf C}}={\bf C}$.
\end{proposition}
\begin{proof}
Since the problem is linear in ${\bf C}$ and
${\bf G}_n\hat{{\bf C}}$ is fixed in $L^2(0,T,{\bf H})$, the proof
is classical (e.g. Dautray and Lions, \cite{Dautray:1988b})
\end{proof}

To insure that $\Theta$ has a fixed point, we show that the Schauder
fixed point theorem can be applied.

\begin{lemma}
\label{lemma:thetacontinu}
The operator $\Theta$ is continuous on $L^2(0,T,{\bf H})$.
\end{lemma}
\begin{proof}
Let
$\hat{{\bf C}}^1$ and $\hat{{\bf C}}^2 \in L^2(0,T,{\bf H})$. 
${\bf C}^1$ and ${\bf C}^2$, the associated solutions to the problem
of proposition \ref{prop:theta}, satisfy,
$$
(\ds \frac{d}{dt} ({\bf C}^1-{\bf C}^2)
, {\bf \phi}) +
a(t,{\bf C}^1-{\bf C}^2,{\bf \phi})= 
({\bf G}_n\hat{{\bf C}}^1-{\bf G}_n\hat{{\bf C}}^2,{\bf \phi}).
$$
Taking ${\bf \phi}={\bf C}^1-{\bf C}^2$ as a test function,
integrating on $[0,t]$, using the coerciveness of $a$ and Cauchy-Schwarz
inequality, we obtain,
$$
\begin{array}{l}
\ds \int_0^t \frac{1}{2}\frac{d }{dt}||{\bf C}^1(s)-{\bf C}^2(s)||^2 + c_0
||{\bf C}^1(s)-{\bf C}^2(s)||_1^2 ds \\
\le \ds \int_0^t
||{\bf G}_n\hat{{\bf C}}^1(s)-{\bf G}_n\hat{{\bf C}}^2(s)||||{\bf C}^1(s)-{\bf C}^2(s)||ds.
\end{array}
$$
As ${\bf C}^1(0)={\bf C}^2(0)={\bf C}^0$ we obtain using Young's inequality,
$$
\begin{array}{l}
||{\bf C}^1(t)-{\bf C}^2(t)||^2 + \ds \int_0^t 2c_0
||{\bf C}^1(s)-{\bf C}^2(s)||_1^2 ds 
\\
\le \ds \int_0^t\frac{1}{\alpha}
||{\bf G}_n\hat{{\bf C}}^1(s)-{\bf G}_n\hat{{\bf C}}^2(s)||^2 ds 

+ \alpha \ds \int_0^t ||{\bf C}^1(s)-{\bf C}^2(s)||^2 ds,
\end{array}
$$
and with $\alpha = 2c_0$,
$$
||{\bf C}^1(t)-{\bf C}^2(t)||^2 \le \ds \int_0^T\frac{1}{2c_0}
||{\bf G}_n\hat{{\bf C}}^1(s)-{\bf G}_n\hat{{\bf C}}^2(s)||^2 ds.
$$
Eventually, we obtain, integrating on $[0,T]$, 
$$
||\Theta\hat{{\bf C}}^1-\Theta\hat{{\bf C}}^2||_{L^2(0,T,{\bf H})}= 
||{\bf C}^1-{\bf C}^2||_{L^2(0,T,{\bf H})} 
\le \ds \sqrt{\frac{T}{2c_0}}
||{\bf G}_n\hat{{\bf C}}^1-{\bf G}_n\hat{{\bf C}}^2||_{L^2(0,T,{\bf H})} 
$$
and $\Theta$ is continuous as ${\bf G}_n$ is
\end{proof}

\begin{lemma}
\label{lemma:thetaborne}
The operator $\Theta$ maps $L^2(0,T,{\bf H})$ in the ball
$$
B=\lbrace {\bf C} \in
L^2(0,T,{\bf H})
, ||{\bf C}||_{L^2(0,T,{\bf H})} \le \ds \sqrt{T(\ds \frac{2LTn^2}{c_0} +
||{\bf C}^0||^2)} \rbrace.
$$
In particular, we have, $\Theta(B) \subset B$.
\end{lemma}
\begin{proof}
Let $\hat{{\bf C}} \in L^2(0,T,{\bf H})$. ${\bf C}$, the 
solution to the problem of proposition \ref{prop:theta}, satisfies,
$$
\ds (\frac{d}{dt} {\bf C}, {\bf \phi}) +
a(t,{\bf C},{\bf \phi})= 
({\bf G}_n\hat{{\bf C}},{\bf \phi}).
$$
Taking ${\bf \phi}={\bf C}$ as a test function, 
integrating $[0,t]$, using the coerciveness of $a$ and the
Cauchy-Schwarz inequality 
we obtain,
$$
\ds \int_0^t \frac{1}{2}\frac{d}{dt}||{\bf C}(s)||^2 + c_0
||{\bf C}(s)||_1^2 ds \le \ds \int_0^t
||{\bf G}_n\hat{{\bf C}}(s)||||{\bf C}(s)||ds,
$$
with Young's inequality, 
$$
||{\bf C}(t)||^2 + \ds \int_0^t 2c_0
||{\bf C}(s)||_1^2 ds \le \ds \int_0^t\frac{1}{\alpha}
||{\bf G}_n\hat{{\bf C}}(s)||^2 ds + 
\alpha \ds \int_0^t ||{\bf C}(s)||^2 ds + ||{\bf C}^0||^2,
$$
with $\alpha = 2c_0$ and as $\ds \int_0^t
||{\bf G}_n\hat{{\bf C}}(s)||^2 ds \le 4LTn^2$, we have
$$
||{\bf C}(t)||^2 \le \ds \frac{4LTn^2}{2c_0} + ||{\bf C}^0||^2,
$$
integrating once more on $[0,T]$ we obtain,
$$
||\Theta\hat{{\bf C}}||_{L^2(0,T,{\bf H})}^2= 
||{\bf C}||_{L^2(0,T,{\bf H})}^2 
\le \ds T(\ds \frac{2LTn^2}{c_0} +
||{\bf C}^0||^2).
$$
\end{proof}

\begin{lemma}
\label{lemma:thetacompact}
The operator $\Theta$ is compact.
\end{lemma}
\begin{proof}
Let $B$ be a bounded set in $L^2(0,T,{\bf H})$. Let us show that 
$\Theta(B)$ is bounded in $W({\bf H}^1)$.
Let $\hat{{\bf C}} \in B \subset L^2(0,T,{\bf H})$ and let
${\bf C}$ be the associated solution to the problem of proposition
\ref{prop:theta}. 
As in the proof of lemma \ref{lemma:thetaborne} we obtain
$$
||{\bf C}(t)||^2 + \ds \int_0^t 2c_0
||{\bf C}(s)||_1^2 ds \le \ds \int_0^t\frac{1}{\alpha}
||{\bf G}_n\hat{{\bf C}}(s)||^2 ds + 
\alpha \ds \int_0^t ||{\bf C}(s)||^2 ds + ||{\bf C}^0||^2.
$$
Taking $\alpha = c_0$ this time, we obtain
$$
\ds \int_0^T ||{\bf C}(s)||_1^2 ds 
\le \ds \frac{1}{c_0}||{\bf G}_n\hat{{\bf C}}||_{L^2(0,T,{\bf H})^2} 
 + ||{\bf C}^0||^2 \le \frac{4LTn^2}{c_0} + ||{\bf C}^0||^2,
$$
and $\Theta \hat{{\bf C}}$ is bounded in $L^2(0,T,{\bf H}^1)$.\\
Moreover we obtain
$$
\forall \phi \in {\bf H}^1, \quad (\ds \frac{d {\bf C}}{dt}, {\bf \phi}) +
a(t,{\bf C},{\bf \phi}) = ({\bf G}_n\hat{{\bf C}},{\bf \phi}).
$$ 
From lemma \ref{lemma:formebi}, 
$$
|(\ds \frac{d {\bf C}}{dt}, {\bf \phi})| \le M_a||{\bf C}||_1
 ||{\bf \phi}||_1 +
||{\bf G}_n\hat{{\bf C}}||||{\bf \phi}|| 
\le 
M_a||{\bf C}||_1
 ||{\bf \phi}||_1 +
2n\sqrt{L}||{\bf \phi}||_1,
$$
and
$$
\ds \int_0^T ||\frac{d{\bf C}}{dt}||_{({\bf H}^1)'}^2 ds \le 2 \ds
\int_0^T (4Ln^2 + M^2 ||{\bf C}||_1^2)ds. 
$$
And therefore, $||\ds \frac{d{\bf C}}{dt}||_{L^2(0,T,({\bf H}^1)')}$ is bounded in 
$L^2(0,T,({\bf H}^1)')$.\\
The range of $\Theta$ is in $W({\bf H}^1)$, from lemma
\ref{lemma:compacite}, the injection $W({\bf H}^1) \subset
L^2(0,T,{\bf H})$ 
is compact, and this concludes the proof
\end{proof}

It is now possible to state the main result of this section, concerning the existence
of weak solutions to the approximated problem.

\begin{theorem}
Let $n>0$ be a fixed integer. Let ${\bf C}^0 \in {\bf H}$. 
There exists a solution, ${\bf C}_n$, to the problem:\\
find ${\bf C} \in W({\bf H}^1)$ such that, 
$$
\forall \phi \in {\bf H}^1, \quad (\ds \frac{d {\bf C}}{dt}, {\bf \phi}) +
a(t,{\bf C},{\bf \phi}) = ({\bf G}_n{\bf C},{\bf \phi}),
$$ 
in the ${\mathcal{D}}'(]0,T[)$ sens,\\ 
and ${\bf C}(0)={\bf C}^0$.\\
\end{theorem}
\begin{proof}
From lemma \ref{lemma:thetacontinu},
\ref{lemma:thetaborne}, \ref{lemma:thetacompact} and the Schauder
fixed-point theorem, the operator $\Theta$ has a fixed point, which
is the solution sought
\end{proof}
  
\subsection{Step 2: letting $n \rightarrow \infty$}
\noindent
We now pass to the limit as $n \rightarrow \infty$, in the equations,
\begin{equation}
\label{equ:probn}
\begin{array}{l}
\forall \phi \in {\bf H}^1, \quad (\ds \frac{d {\bf C}_n}{dt}(t),\phi) 
+ a(t,{\bf C}_n,\phi) = ({\bf G}_n{\bf C}_n(t),\phi),\\
{\bf C}_n(0)={\bf C}^0.
\end{array}
\end{equation}

This is achieved in two steps:
\begin{itemize}
\item[a)] a priori estimations on the sequence ${\bf C}_n$,
\item[b)] extraction of subsequences and letting $n \rightarrow \infty$. 
\end{itemize} 

\begin{itemize}
\item[a)] {\bf estimations}.
Let us show that:
\begin{itemize}
\item[(a.1)]the sequence $({\bf C}_n)_{n > 0}$ is bounded in $L^{\infty}(0,T,{\bf H})$,
\item[(a.2)]the sequence  $({\bf C}_n)_{n > 0}$ is bounded in $L^{2}(0,T,{\bf H}^1)$,
\item[(a.3)]the sequence  $(\ds \frac{d {\bf C}_n}{dt})_{n > 0}$ is
bounded in $L^{2}(0,T,({\bf H}^1)')$.
\end{itemize}
Taking, ${\bf C}_n$ as a test function in (\ref{equ:probn}) we obtain,
$$
\frac{1}{2}\frac{d}{dt}||{\bf C}_n||^2 + a(t,{\bf C}_n,{\bf C}_n)=({\bf G}_n{\bf C}_n,{\bf C}_n),
$$
or,
$$
\frac{1}{2}\frac{d}{dt}||{\bf C}_n||^2 + c_0||{\bf C}_n||_1^2 
\le ||{\bf G}_n{\bf C}_n||||{\bf C}_n|| \le M_g ||{\bf C}_n||^2,
$$
and integrating on $[0,t]$, we obtain
\begin{equation}
\label{eqn:estim}
||{\bf C}_n||^2 + 2c_0\ds \int_0^t ||{\bf C}_n||_1^2 ds 
\le 2 M_g \ds \int_0^t ||{\bf C}_n||^2 ds + ||{\bf C}^0||^2.
\end{equation}

Equation (\ref{eqn:estim}) gives
$$
||{\bf C}_n||^2 \le 2 M_g \ds \int_0^t ||{\bf C}_n||^2 ds + ||{\bf C}^0||^2.
$$
Using Gronwall lemma, we have
\begin{equation}
\label{eqn:estim2}
||{\bf C}_n(t)||^2 \le ||{\bf C}^0||^2 \exp(2 M_g T),
\end{equation}
and the sequence $({\bf C}_n)_{n > 0}$ is bounded in $L^{\infty}(0,T,{\bf H})$.\\

Equation (\ref{eqn:estim}) also gives
$$
\ds \int_0^t ||{\bf C}_n||_1^2 ds 
\le \frac{M_g}{c_0} \ds \int_0^t ||{\bf C}_n||^2 ds + \frac{1}{2c_0}||{\bf C}^0||^2,
$$
and with (\ref{eqn:estim2})
$$
\ds \int_0^t ||{\bf C}_n||_1^2 ds 
\le \frac{M_gT}{c_0} ||{\bf C}^0||^2 \exp(2 M_g T)
+ \frac{1}{2c_0}||{\bf C}^0||^2.
$$
Therefore the sequence $({\bf C}_n)_{n > 0}$ is bounded in $L^{2}(0,T,{\bf H}^1)$.\\

Let us now give an estimation for the sequence, $(\ds 
\frac{d {\bf C}_n}{dt})_{n > 0}$. 
We have
$$
|(\frac{d {\bf C}_n}{dt},\phi)| \le |a(t,{\bf C}_n,\phi)| + |({\bf G}_n{\bf C}_n,\phi)|,
$$
and therefore
$$
|(\frac{d {\bf C}_n}{dt},\phi)| \le M_a||{\bf C}_n||_1||\phi||_1 + M_g||{\bf C}_n||_1||\phi||_1,
$$
that is to say
$$
\begin{array}{l}
\ds \int_0^T ||\frac{d {\bf C}_n}{dt}||_{({\bf H}^1)'}^2 
 \le 2(M_a^2 +M_g^2) \ds \int_0^T ||{\bf C}_n||_1^2 \\[10pt]
\le 2(M_a^2 +M_g^2)(\ds \frac{M_gT}{c_0} ||{\bf C}^0||^2 \exp(2 M_g T)
+ \ds \frac{1}{2c_0}||{\bf C}^0||^2), 
\end{array}
$$
and the sequence $(\ds \frac{d {\bf C}_n}{dt})_{n > 0}$ is bounded in
$L^{2}(0,T,({\bf H}^1)')$.

\item[b)]{\bf passing to the limit}.
Let us first recall that ${\mathcal{D}}(]0,T[,{\bf H}^1) \subset W({\bf H}^1)$. 
Therefore $\forall \phi \in {\bf H}^1$ and $\forall \varphi \in
{\mathcal{D}}(]0,T[)$, we have $\psi = \phi \otimes \varphi \in
L^2(0,T,{\bf H}^1)$ and 
$\ds \frac{d \psi}{dt} \in L^2(0,T,({\bf H}^1)')$. 

\begin{itemize}
\item[b.1)] The term $(\ds \frac{d {\bf C}_n}{dt},\phi)$: 
from (a.3), we are able to extract from the sequence $(\ds \frac{d {\bf C}_n}{dt})_{n > 0}$ 
a subsequence (denoted in the same way) converging to some ${\bf h}$
in $L^2(0,T,({\bf H}^1)')$ weak star, that is to say, for all $\phi \in {\bf H}^1$ 
and all $\varphi \in {\mathcal{D}}(]0,T[)$,
$$
\lim_{n \rightarrow \infty} 
\ds \int_0^T (\ds \frac{d {\bf C}_n}{dt},\phi)\varphi ds =  
\ds \int_0^T ({\bf h},\phi)\varphi ds. 
$$
Moreover, by definition, we obtain
$$
\ds \int_0^T (\ds \frac{d {\bf C}_n}{dt},\phi)\varphi ds =  
- \ds \int_0^T ({\bf C}_n,\phi)\ds \frac{d \varphi}{dt} ds. 
$$
From (a.2), we are able to extract from the sequence $({\bf C}_n)_{n >
0}$ a subsequence (denoted in the same way) converging to some
${\bf C}$ in $L^2(0,T,{\bf H}^1)$ weak. 
Therefore, for all $\phi \in {\bf H}^1$ and all $\varphi \in {\mathcal{D}}(]0,T[)$,
$$
\lim_{n \rightarrow \infty} 
- \ds \int_0^T ({\bf C}_n,\phi)\ds \frac{d \varphi}{dt} ds 
=- \ds \int_0^T ({\bf C},\phi)\ds \frac{d \varphi}{dt} ds, 
$$
and ${\bf h} = \ds \frac{d {\bf C}}{dt}$ in $L^2(0,T,({\bf H}^1)')$.

\item[b.2)] The term $a(t,{\bf C}_n,\phi)$: 
from (b.1), we can suppose that the sequence $({\bf C}_n)_{n>0}$ converges 
to ${\bf C}$ in $L^2(0,T,{\bf H}^1)$ weak. Therefore the sequence 
$(\partial_x {\bf C}_n)_{n>0}$ converges 
to $\partial_x {\bf C}$ in $L^2(0,T,{\bf H})$ weak. Then, 
for all $\phi \in {\bf H}^1$ and all $\varphi \in {\mathcal{D}}(]0,T[)$,
$$
\lim_{n \rightarrow \infty} \ds \int_0^T a(s,{\bf C}_n,\phi)\varphi ds 
=\ds \int_0^T a(s,{\bf C},\phi)\varphi ds.
$$

\item[b.3)] The term $({\bf G}_n{\bf C}_n,\phi)$: 
from (a.2), (a.3), and from the compacity of the injection 
$W({\bf H}^1) \rightarrow L^2(0,T,{\bf H})$, we can suppose that the
sequence $({\bf C}_n)_{n > 0}$ converges to ${\bf C}$ in
$L^2(0,T,{\bf H})$ strong. 
Therefore, each ${\bf C}_{n,i}$, $i=1, ...4$, converges to
${\bf C}_i$ in $L^2(0,T,L^2(0,L))$ strong. From the inverse Lebesgue
theorem (Br\'ezis, \cite{Brezis:1992}, theorem IV.9. page 58), we
can suppose that:
\begin{itemize}
\item[(b.3.1)] the sequences $({\bf C}_{n,i})_{n>0}$, $i=1, ...4$, converge 
to ${\bf C}_i$ a.e. in $]0,T[ \times ]0,L[$.
\item[(b.3.2)] for $i=1$ to $4$, $|{\bf C}_{n,i}| \le h_i$, $\forall
n>0$, a.e. in  $]0,T[ \times ]0,L[$ and $h_i \in L^2(0,T,L^2(0,L))$.
\end{itemize}
As ${\bf g}_{n,i}(t,x,{\bf C})$ is continuous in its third variable,
we deduce from (b.3.1) that $\forall \phi_i \in H^1(0,L)$ and $\forall \varphi \in 
{\mathcal{D}}(]0,T[)$,
$$
\begin{array}{l}
u_{n,i}(t,x)={\bf g}_{n,i}(t,x,{\bf C}_n(t,x))\phi_i(x)\varphi(t)\\ 

\mbox{\raisebox{-2ex}{$\stackrel{\longrightarrow}{n \rightarrow \infty}$}}

{\bf g}_{i}(t,x,{\bf C}(t,x))\phi_i(x)\varphi(t),
\end{array}
$$ 
a.e. in $]0,T[ \times ]0,L[$.\\
Moreover, from lemma 3.3 and with (b.3.2) we have,
$$
|u_{n,i}| \le M_i ( \ds \sum_{i=1}^4 h_i) |\phi_i||\varphi| \in L^1(]0,T[ \times ]0,L[),
$$
where the $M_i$ are constants. Thus, from the Lebesgue theorem on
dominated convergence, we obtain, 
$$
\lim_{n \rightarrow \infty}
\ds \int_0^T \int_0^L u_{n,i}(t,x)dxdt
= \ds \int_0^T \int_0^L {\bf g}_{i}(t,x,{\bf C}(t,x))\phi_i(x)\varphi(t)dxdt,    
$$
and finally, for all $\phi \in {\bf H}^1$ and all $\varphi \in {\mathcal{D}}(]0,T[)$,
$$
\lim_{n \rightarrow \infty}
\ds \int_0^T ({\bf G}_n{\bf C}_n,\phi)\varphi 
=
\ds \int_0^T ({\bf G}{\bf C},\phi)\varphi. 
$$
\end{itemize}
\end{itemize}
This concludes the proof of the existence of weak solutions 
to the one-dimensional NPZD-model.

\vspace*{1pt}\textlineskip	%) USE THIS MEASUREMENT WHEN THERE IS
\section{Positivity}
\vspace*{-0.5pt}
\noindent
In this section we prove the second part of Theorem \ref{theo:main}:
if initial conditions $N^0,P^0,Z^0$ and $D^0$ are positive then
solutions to the one-dimensional NPZD-model are positive for a.e. $t
\in [0,T]$. To prove this, we need to treat each of the four
equations seperately, in detail, and in a convenient order. We first show that
$Z$ and $P$ are positive. Next we show that $D$ is positive using the
fact that $Z$ and $P$ are positive. Finally, as $Z$, $P$ and $D$ are
positive we obtain the positivity of $N$.

\begin{itemize}
\item 
Let us recall that for all ${\bf C} \in
{\bf H}^1$ and all $t \in [0,T]$, $a_N(t,N,N) \ge 0$, $a_P(t,P,P) \ge
0$, $a_Z(t,Z,Z) \ge 0$ and $a_Z(t,Z,Z) \ge 0$.

\item
$Z$ is positive:\\
Let ${\bf C}$ be a weak solution to the NPZD-model. 
Let us take 
$$
-Z^-=-max(0,-Z),
$$ 
as a test function. 
Since, 
$$
\displaystyle \int_0^L\ds \frac{\partial Z(t,x)}{\partial t}Z^-(t,x) dx
=-\ds \frac{1}{2}\ds \frac{d
}{dt}||Z^-(t)||_{L^2(0,L)}^2,
$$
and
$$ 
a_Z(t,Z(t),-Z^-(t))=a_Z(t,Z(t)^-,Z(t)^-),
$$
we obtain
$$
\ds \frac{1}{2}\ds
\frac{d}{dt}||Z(t)^-||_{L^2(0,L)}^2+a_Z(t,Z(t)^-,Z(t)^-)
=-(g_Z({\bf C}(t)),Z(t)^-). 
$$
Let us detail the term $(g_Z({\bf C}),Z^-)$.
$$
\begin{array}{ll}
(g_Z({\bf C}),Z^-)_{L^2(0,L)}=&
\displaystyle \int_0^L(a_p\ds \frac{g_zP^2}{k_z+P^2}ZZ^- 
+ a_d\ds \frac{g_zD^2}{k_z+D^2}ZZ^- 
-m_zZZ^- \\[11pt]
&-\mu_zZZ^-)\ \fc_{]0,l]}
+(-\tau ZZ^-)\ \fc_{]l,L[} + \lambda ZZ^-.\\   
\end{array}
$$
As $ZZ^-=-(Z^-)^2$, we have
$$
\begin{array}{ll}
(g_Z({\bf C}),Z^-)_{L^2(0,L)}=&
\displaystyle \int_0^L(-(a_p\ds \frac{g_zP^2}{k_z+P^2})(Z^-)^2 
- (a_d\ds \frac{g_zD^2}{k_z+D^2})(Z^-)^2 
+m_z(Z^-)^2\\[11pt]
& +\mu_z(Z^-)^2)\ \fc_{]0,l]} +(\tau (Z^-)^2)\ \fc_{]l,L[} -\lambda (Z^-)^2,\\   
\end{array}
$$
and
$$
\begin{array}{ll}
(g_Z({\bf C}),Z^-)_{L^2(0,L)} \ge & \displaystyle \int_0^L(-(a_p\ds \frac{g_zP^2}{k_z+P^2})(Z^-)^2 \\[11pt]
& - (a_d\ds \frac{g_zD^2}{k_z+D^2})(Z^-)^2)\fc_{]0,l]}-\lambda (Z^-)^2,\\      
\end{array}
$$
or
$$
\begin{array}{ll}
-(g_Z({\bf C}),Z^-)_{L^2(0,L)} \le &\displaystyle \int_0^L((a_p\ds \frac{g_zP^2}{k_z+P^2})(Z^-)^2 \\[11pt]
& +(a_d\ds \frac{g_zD^2}{k_z+D^2})(Z^-)^2)+\lambda (Z^-)^2.     
\end{array}
$$
Thus
$$
\begin{array}{ll}
-(g_Z({\bf C}),Z^-)_{L^2(0,L)} & \le \displaystyle \int_0^L(g_z(a_p+a_d)+\lambda)(Z^-)^2),\\[10pt]
                               & =(g_z(a_p+a_d)+\lambda)||Z^-||_{L^2(0,L)}^2.\\   
\end{array}
$$
As $a_Z(t,Z^-,Z^-) \ge 0$, we obtain
$$
\ds \frac{d}{dt}||Z^-||_{L^2(0,L)}^2 \le 2(g_z(a_p+a_d)+\lambda)||Z^-||_{L^2(0,L)}^2.
$$
Integrating this inequality on $[0,t]$ and using Gronwall's lemma, we obtain
$$
||Z^-(t)||_{L^2(0,L)}^2 \le ||Z^-(0)||_{L^2(0,L)}^2 \exp(2(g_z(a_p+a_d)+\lambda)t).    
$$
Therefore $Z$ is positive.

\item
$P$ is positive:\\
In the same manner, let us examine the term $(g_P({\bf C}),P^-)_{L^2(0,L)} $.\\
$$
\begin{array}{ll}
(g_P({\bf C}),P^-)_{L^2(0,L)} &= \displaystyle \int_0^L
(\mu_p(1-\gamma)L_IL_NPP^- -(\ds \frac{g_z P^2}{k_z+P^2}ZP^-) \\[11pt]
&-m_pPP^-)\ \fc_{]0,l]} + (-\tau PP^-)\ \fc_{]l,L[} + \lambda PP^-,\\[7pt]
& =\displaystyle \int_0^L 
(-\mu_p(1-\gamma)L_IL_N(P^-)^2 -(\ds \frac{g_z P^2}{k_z+P^2}ZP^-)\\[11pt]
& +m_p(P^-)^2)\ \fc_{]0,l]} + (\tau (P^-)^2)\ \fc_{]l,L[} - \lambda (P^-)^2,\\[7pt]
& \ge \displaystyle \int_0^L (-\mu_p(1-\gamma)L_IL_N(P^-)^2 \\[11pt]
&-(\ds \frac{g_z P^2}{k_z+P^2}ZP^-))\ \fc_{]0,l]}- \lambda (P^-)^2,
\end{array}
$$
and
$$
\begin{array}{ll}
-(g_P({\bf C}),P^-)_{L^2(0,L)} & \le \displaystyle \int_0^L
\mu_p(1-\gamma)L_IL_N(P^-)^2 \\[11pt]
& +(\ds \frac{-g_z P}{k_z+P^2}Z(P^-)^2)\ \fc_{]0,l]}+ \lambda (P^-)^2.
\end{array}
$$
The  function $x \mapsto \ds \frac{-x}{k_z+x^2}$ is bounded by
$\ds \frac{1}{2\sqrt{k_z}}$ on $\Rr$.\\ 
$L_N$ and $L_I$ are bounded by $1$.\\
$Z(t)$, is a solution to the NPZD-model and therefore belongs to 
$H^1(0,L) \subset L^{\infty}(0,L)$. Hence we have, $\forall t,\ Z(t)
\le ||Z(t)||_{\infty}$, and
$$ 
\begin{array}{l}
-(g_P({\bf C}),P^-)_{L^2(0,L)}
\le
 (\lambda+\mu_p(1-\gamma)+g_z\ds \frac{1}{2\sqrt{k_z}}||Z(t)||_{\infty})
||P^-(t)||_{L^2(0,L)}^2.
\end{array}
$$
We conclude in the same way to obtain
$$
\begin{array}{ll}
||P^-(t)||_{L^2(0,L)}^2 &\le ||P^-(0)||_{L^2(0,L)}^2 \exp(\displaystyle \int_0^t 2(\lambda+\mu_p(1-\gamma) \\[11pt]
&+g_z\ds \frac{1}{2\sqrt{k_z}}||Z(s)||_{\infty})ds).    
\end{array}
$$

\item
$D$ is positive:\\
$$
\begin{array}{ll}
(g_D({\bf C}),D^-)_{L^2(0,L)}& = \displaystyle \int_0^L((1-a_p)(\ds \frac{g_z P^2}{k_z+P^2}ZD^-)
-a_d(\ds \frac{g_z D^2}{k_z+D^2}ZD^-) \\[11pt]
& +m_pPD^-+m_zZD^-
+\mu_d(D^-)^2)\ \fc_{]0,l]}\\[8pt]
&+(\tau (D^-)^2)\ \fc_{]l,L[} +\lambda DD^-.
\end{array}
$$
Because $P,Z$ and $D^-$ are positive, we obtain
$$
\begin{array}{ll}
(g_D({\bf C}),D^-)_{L^2(0,L)}  &\ge \displaystyle \int_0^L(-a_d(\ds \frac{g_z
D^2}{k_z+D^2}ZD^-))\ \fc_{]0,l]} -\lambda (D^-)^2,\\[11pt]
-(g_D({\bf C}),D^-)_{L^2(0,L)}  &\le
\displaystyle \int_0^L(-a_d(\ds \frac{g_zD}{k_z+D^2}Z(D^-)^2))\ \fc_{]0,l]}+\lambda (D^-)^2,\\[11pt]
-(g_D({\bf C}),D^-)_{L^2(0,L)} &\le
\displaystyle \int_0^L(a_d(g_z\ds \frac{1}{2\sqrt{k_z}}||Z||_{\infty})(D^-)^2)\ \fc_{]0,l]}+\lambda (D^-)^2,\\[11pt]
-(g_D({\bf C}),D^-)_{L^2(0,L)}  &\le
(\lambda+ a_d(g_z\ds \frac{1}{2\sqrt{k_z}}||Z||_{\infty}))||D^-||_{L^2(0,L)}^2.\\
\end{array}
$$
Hence
$$
\ds \frac{d}{dt}||D^-||_{L^2(0,L)}^2 \le 2( \lambda + a_d(g_z\ds \frac{1}{2\sqrt{k_z}}||Z||_{\infty}))||D^-||_{L^2(0,L)}^2,
$$
and we can conclude.

\item
$N$ is positive:\\
$$
\begin{array}{ll}
(g_N({\bf C}),N^-)_{L^2(0,L)} &=\displaystyle \int_0^L
(-\mu_p(1-\gamma) L_I L_{N}PN^-+\mu_z ZN^- \\[11pt] 
&+ \mu_d DN^-)\ \fc_{]0,l]} +(\tau(P+Z+D)N^-)\ \fc_{]l,L[}+\lambda NN^-.
\end{array}
$$
Because $P,Z,D$ and $N^-$ are positive, we have
$$
\begin{array}{l}
(g_N({\bf C}),N^-)_{L^2(0,L)}
\ge \displaystyle \int_0^L
(-\mu_p(1-\gamma) L_I L_{N}PN^-)\ \fc_{]0,l]}-\lambda (N^-)^2,\\[8pt]
-(g_N({\bf C}),N^-)_{L^2(0,L)}
\le \displaystyle \int_0^L
(-\mu_p(1-\gamma) L_I \ds \frac{1}{k_n+|N|}P(N^-)^2)\ \fc_{]0,l]} +\lambda (N^-)^2.
\end{array}
$$
Once again we can conclude and the proof of theorem \ref{theo:main} is
complete.
\end{itemize}

Hence, if initial concentrations are positive then concentrations are
always positive and both models, with or without absolute values in the
nonlinear terms, are equivalent. 

\vspace*{1pt}\textlineskip	%) USE THIS MEASUREMENT WHEN THERE IS
\section{Existence, positivity and uniqueness 
for different $G_P$, $G_D$, $L_I$ and zooplankton mortality formulations}
\label{sec:uniqueness}
\vspace*{-0.5pt}
\noindent
Functions used to parameterize biological fluxes such as zooplankton
grazing on phytoplankton, $G_P$, or on detritus, $G_D$, light limited
growth rate, $L_I$ or zooplankton mortality (which is a constant,
$m_z$, in our model), vary from one modeling study to another. 
One can wonder if the existence result still applies with these different
formulations. To answer this question it should be noticed that the
key argument used in the proof is the fact that the nonlinear reaction terms
allow us to define a nonlinear continuous operator ${\bf G}$ satisfying
$||{\bf GC}||_{L^2(0,T,{\bf H})} \le M_g
||{\bf C}||_{L^2(0,T,{\bf H})}$. Therefore, as all the functions listed in 
Table \ref{tab:functions}, found in the
literature, are continuous and
bounded on $\Rr^+$ or $(\Rr^+)^2$, the existence result stays
correct. Positivity can also easily be checked for all these different
formulations. It should however be mentioned that some studies use a
quadratic zooplankton mortality term which can not be treated with the
method we propose. 

\begin{table}[htbp]
\caption{ Different parameterizations found in the literature. All
parameters are positive constants. \label{tab:functions} }
\begin{tabular}{|p{2cm}|p{6cm}|p{3cm}|}
\hline
&&\\[-5pt]
$Z$ grazing on $P$, $G_P$  & $\ds \frac{g_zP^2}{k_z + P^2}$ & this\ study\ and\
e.g.\ Fennel\ \ea \cite{Fennel:2001}\\
\cline{2-3}
&&\\[-5pt]
& $\ds \frac{g_zP^2}{k_z + P^2 + D^2}$ & e.g.\ Losa \ea \cite{Losa:2001}\\
&&\\[-5pt]
\cline{2-3}
&&\\[-5pt] 
 &  $\ds\frac{g_zrP^2}{k_z(rP+(1-r)D)+rP^2+(1-r)D^2}$ & e.g.\ Fasham
\ea \cite{Fasham:1990} \\
&&\\[-5pt]
\hline
&&\\[-5pt]
$Z$ grazing on $D$, $G_D$  & $\ds \frac{g_zD^2}{k_z + D^2}$ & this\ study\ and\
e.g.\ Fennel\ \ea \cite{Fennel:2001}\\
&&\\[-5pt]
\cline{2-3}
&&\\[-5pt]
& $\ds \frac{g_zD^2}{k_z + P^2 + D^2}$ & e.g.\ Losa\ \ea \cite{Losa:2001}\\
&&\\[-5pt]
\cline{2-3}
&&\\[-5pt]
& $\ds \frac{g_zrD^2}{k_z(rP+(1-r)D)+rP^2+(1-r)D^2}$ & e.g.\ Fasham\
\ea \cite{Fasham:1990} \\
&&\\[-5pt]
\hline
&&\\[-5pt]
light\ limited\ growth\ rate, $L_I$ & $1-\exp(-PAR(t,x,P)/k_{par})$
& this\ study\ and\ e.g.\ L\'evy\ \ea \cite{Levy:1998a} \\
\cline{2-3}
&&\\[-5pt]
& $\ds \frac{v_p \alpha PAR(t,x,P)}{(v_p^2 + \alpha^2
PAR(t,x,P)^2)^{1/2}}$ 
& e.g.\ Spitz\ \ea \cite{Spitz:1998}\\
&&\\[-5pt]
\hline
&&\\[-5pt]
Z\ mortality& $m_z$ & this\ study\ and\ e.g.\ L\'evy \ea
\cite{Levy:1998b}\\
\cline{2-3}
&&\\[-5pt]
            & $\ds \frac{m_z Z}{k + Z}$ & e.g.\ Losa\
\ea \cite{Losa:2001}\\
&&\\
\hline
\end{tabular}
\end{table}

Let us now concentrate on the question of the uniqueness of weak 
solutions to the one-dimensional NPZD-model.\\ 
In order to prove uniqueness we need the nonlinear reaction terms to satisfy a local 
Lipschitz condition which was not needed to obtain the existence result. To verify that 
such a condition holds we examine in some details the optical model from which the 
$PAR(t,x,P)$ term and consequently the $L_I(t,x,P)$ term are calculated.\\ 
In the different equations $L_I(t,x,P)$ allways appears in the product form 
$PL_I(t,x,P)$. Concerning this product the desired local Lipschitz 
condition reads as follows:\\
~\\
for all $(t,x) \in [0,T] \times [0,L]$, and all $P,\hat{P} \in [0,+\infty[$,   
\begin{equation}
\label{lipschitz}
|PL_I(t,x,P)-\hat{P}L_I(t,x,\hat{P})| \le K_I(P,\hat{P}) |P-\hat{P}|,
\end{equation}
where $K_I$ is a continuous nonnegative real-valued function which is increasing in each 
variable.\\
~\\
In the optical model we considered, two different wavelengths are
taken into account and the absorption coefficients 
depend on the local phytoplankton concentrations:
$$
\begin{array}{ll}
PAR(t,x,P)=&Q(t)(\exp(-(k_{go}+k_{gp}(\ds \frac{12 P
r_d}{r_{pg}r_c})^{l_g}) x) \\[7pt]
&+\exp(-(k_{ro}+k_{rp}(\ds \frac{12 P r_d}{r_{pg}r_c})^{l_r}) x)).
\end{array}
$$
$Q(t)$ is proportional to the irradiance intensity hitting the sea surface 
at time $t$. Parameters are given in Table \ref{tab:optic}. 
Let us suppose that $Q(t) \in L^{\infty}(0,T)$ and that
$Q(t) \ge 0$. 
Even though the exponents $l_g$ and $l_r$ satisfy $0< l_g < 1$ and $0< l_r < 1$, 
an easy calculation of the derivative, $\ds  \frac{d}{dP}(PL_I(t,x,P))$, 
shows that with such an optical model property (\ref{lipschitz}) is satisfied with:
$$
K_I(P,\hat{P})= 1 + \ds \frac{||Q||_{\infty} L}{k_{par}}
(k_{gp} l_g \ds (\frac{12r_d}{r_{pg}r_{c}})^{l_g} (\max{(P,\hat{P})})^{l_g}
+k_{rp} l_r \ds (\frac{12r_d}{r_{pg}r_{c}})^{l_r} (\max{(P,\hat{P})})^{l_r}).
$$
In the literature $PAR(t,x,P)$ is often parameterized by,
$$
PAR(t,x,P)=Q(t)\exp(-(k_1 + k_2 P)x),
$$
where $k_1$ and $k_2$ are positive constants. With this simpler formulation 
property (\ref{lipschitz}) is clearly satisfied.\\
The following two lemmas give the local Lipschitz property satisfied by all four reaction terms 
of the NPZD-model. 
\begin{lemma}
\label{lemma:besogne2}
The nonlinear reaction terms $g_N$, $g_P$, $g_Z$ and $g_D$ 
satisfy:\\
for all $(t,x) \in [0,T] \times [0,L]$, 
and all ${\bf C},\hat{{\bf C}} \in (\Rr^+)^4$,  
$$
\begin{array}{l}
|g_N(t,x,{\bf C}) - g_N(t,x,\hat{{\bf C}})| \le K_N(P,\hat{P})(|N-\hat{N}|+|P-\hat{P}|+|Z-\hat{Z}|+|D-\hat{D}|),\\[4pt]
|g_P(t,x,{\bf C}) - g_P(t,x,\hat{{\bf C}})| \le K_P(P,\hat{P})(|N-\hat{N}|+|P-\hat{P}|+|Z-\hat{Z}|+|D-\hat{D}|),\\[4pt]
|g_Z(t,x,{\bf C}) - g_Z(t,x,\hat{{\bf C}})| \le K_Z(Z,\hat{Z})(|N-\hat{N}|+|P-\hat{P}|+|Z-\hat{Z}|+|D-\hat{D}|),\\[4pt]
|g_D(t,x,{\bf C}) - g_D(t,x,\hat{{\bf C}})| \le K_D(Z,\hat{Z})(|N-\hat{N}|+|P-\hat{P}|+|Z-\hat{Z}|+|D-\hat{D}|),
\end{array}
$$
where $K_N$, $K_P$, $K_Z$, $K_D$ are continuous nonnegative real-valued functions which are increasing in each 
variable.
\end{lemma}
\begin{proof}
Functions $l(x)=\ds \frac{x}{k_n+x}$ and $g(x)=\ds\frac{x^2}{k_z^2 +x^2}$ are
continuously differentiable on $[0,+\infty[$, and 
$$
|l'(x)| \le \ds \frac{1}{k_n},\quad |g'(x)| \le \ds \frac{3\sqrt{3}}{8\sqrt{k_z}}.
$$
Therefore $l$ and $g$ are Lipschitz continuous.\\
It is clear that
$$
\begin{array}{ll}
|g_N(t,x,{\bf C}) - g_N(t,x,\hat{{\bf C}})| & \le \mu_p(1-\gamma)
|L_I(t,x,P)PL_N(N) - L_I(t,x,\hat{P})\hat{P}L_N(\hat{N})|\\[4pt] 
&+ (\mu_z + \tau)|Z-\hat{Z}|+ (\mu_d + \tau)|D-\hat{D}| \\[4pt]
&+\tau |P-\hat{P}| + \lambda |N-\hat{N}|.  
\end{array}
$$
Now since 
$$
\begin{array}{ll}
|L_I(t,x,P)PL_N(N) - L_I(t,x,\hat{P})\hat{P}L_N(\hat{N})|=&|L_I(t,x,P)P(L_N(N) - L_N(\hat{N}))\\[4pt] 
&+ L_N(\hat{N})( L_I(t,x,P)P) - L_I(t,x,\hat{P})\hat{P}|,
\end{array}
$$
we have
$$
|L_I(t,x,P)PL_N(N) - L_I(t,x,\hat{P})\hat{P}L_N(\hat{N})| \le 
\frac{1}{k_n} \max{(P,\hat{P})} 
|N - \hat{N}| 
+ K_I(P,\hat{P})|P - \hat{P}|. 
$$
We then define
$$
K_N(P,\hat{P})=\max{(
\lambda + \mu_p(1-\gamma)\frac{1}{k_n} \max{(P,\hat{P})} ,
\tau + \mu_p(1-\gamma)K_I(P,\hat{P}),
\mu_z + \tau,
\mu_d + \tau)}.
$$
$K_P$, $K_Z$ and $K_D$ are obtained in the same way
\end{proof}

\begin{lemma}
\label{prop:biendefbis}
For $t \in [0,T]$ and for positive ${\bf C}(t), \hat{{\bf C}}(t) \in {\bf H}^1$,
there exists a constant $L_{\infty}$, depending on $||{\bf C}(t)||_{\infty}$
and $||\hat{{\bf C}}(t)||_{\infty}$, such that the operator, 
${\bf G}$, satisfies, 
$$
||{\bf GC}(t) - {\bf G}\hat{{\bf C}}(t)|| \le
  L_{\infty}||{\bf C}(t)-\hat{{\bf C}}(t)||.
$$
\end{lemma}
\begin{proof}
From lemma \ref{lemma:injection}, 
${\bf H}^1 \subset {\bf L}^{\infty}$.
From lemma \ref{lemma:besogne2} we have, 
$$
\begin{array}{l}
||{\bf G}{\bf C}(t)-{\bf G}\hat{{\bf C}}(t)||^2 \\[7pt] 
=  \displaystyle \int_{0}^{L}
|g_N(t,x,{\bf C}(t,x))-g_N(t,x,\hat{{\bf C}}(t,x))|^2
+|g_P(t,x,{\bf C}(t,x))-g_P(t,x,\hat{{\bf C}}(t,x))|^2\\[7pt] 
+|g_Z(t,x,{\bf C}(t,x))-g_Z(t,x,\hat{{\bf C}}(t,x))|^2
+|g_D(t,x,{\bf C}(t,x))-g_D(t,x,\hat{{\bf C}}(t,x))|^2 dx,\\[7pt] 
 \le cte (K_N(||P(t)||_{\infty},||\hat{P}(t)||_{\infty}))^2\ [\ ||N(t)-\hat{N}(t)||_{L^2(0,L)}^2\\[7pt] 
 +||P(t)-\hat{P}(t)||_{L^2(0,L)}^2 + ||Z(t)-\hat{Z}(t)||_{L^2(0,L)}^2 + ||D(t)-\hat{D}(t)||_{L^2(0,L)}^2\ ]\\[7pt] 
 + ...\\[7pt] 
 \le L_{\infty}(||{\bf C}(t)||_{\infty},||\hat{{\bf C}}(t)||_{\infty})^2 
|| {\bf C}(t) - \hat{{\bf C}}(t) ||^2.
\end{array}
$$
\end{proof}

Elementary calculations show that functions of
Table \ref{tab:functions} 
are continuously differentiable on $\Rr^+$ or $(\Rr^+)^2$, 
with bounded first derivatives. Therefore they are
Lipschitz continuous and the uniqueness result presented below also
holds for these formulations.

\begin{proposition}
\label{prop:unicite}
The weak solution to the
one-dimensional NPZD-model, ${\bf C} \in W({\bf H}^1)$, is unique.
\end{proposition}
\begin{proof}
Let us suppose that there are two solutions ${\bf C}^1$ and 
${\bf C}^2 \in W({\bf H}^1)$. 
They satisfy
$$
\forall \phi \in {\bf H}^1, \quad (\ds \frac{d}{dt} ({\bf C}^1 - {\bf C}^2),
{\bf \phi}) +
a(t,{\bf C}^1 - {\bf C}^2,{\bf \phi}) = ({\bf GC}^1-{\bf GC}^2,{\bf \phi}),
$$ 
Let us choose $\phi ={\bf C}^1 - {\bf C}^2$ as a test function. 
Using the coerciveness of $a$ and the Cauchy-Schwarz inequality, we obtain
$$
\ds \frac{1}{2}\frac{d}{dt}||{\bf C}^1(t) - {\bf C}^2(t)||^2
+c_0 ||{\bf C}^1(t) - {\bf C}^2(t)||_1^2
\le 
||{\bf GC}^1(t) - {\bf GC}^2(t)||||{\bf C}^1(t) - {\bf C}^2(t)||.
$$
With Young's inequality, we obtain
$$
\begin{array}{l}
\ds \frac{1}{2}\frac{d}{dt}||{\bf C}^1(t) - {\bf C}^2(t)||^2
+c_0 ||{\bf C}^1(t) - {\bf C}^2(t)||_1^2 \\[7pt]
\le 
\ds \frac{1}{2\alpha}||{\bf GC}^1(t) - {\bf GC}^2(t)||^2
+\ds \frac{\alpha}{2}||{\bf C}^1(t) - {\bf C}^2(t)||^2,
\end{array}
$$
with $\alpha = 2c_0$,
$$
\ds \frac{d}{dt}||{\bf C}^1(t) - {\bf C}^2(t)||^2
\le 
\frac{1}{2 c_0}||{\bf GC}^1(t) - {\bf GC}^2(t)||^2.
$$
From lemma (\ref{prop:biendefbis}) we obtain
$$
\ds \frac{d}{dt}||{\bf C}^1(t) - {\bf C}^2(t)||^2
\le 
\ds \frac{1}{2c_0} L_{\infty}^2(t)||{\bf C}^1(t)-{\bf C}^2(t)||^2.
$$
Thus, integrating on $[0,t]$ and using Gronwall's lemma
$$
||{\bf C}^1(t)-{\bf C}^2(t)||^2 \le
||{\bf C}^1(0)-{\bf C}^2(0)||^2  \exp(\ds \int_0^t 
\frac{1}{2c_0}L_{\infty}^2(s) ds).
$$
This concludes the proof
\end{proof}

\begin{table}[htbp]
\caption{ Optical model parameters \label{tab:optic} }
\begin{tabular}{|p{5cm}|c|c|c|}
\hline
parameter& name & value & unit \\
\hline
Redfield\ ratio\ C:N & $r_d$ & 6.625 & \\ 
contribution\ of\ Chl\ to\ absorbing\ pigments & $r_{pg}$ & 0.7 & \\
carbone:chlorophyll ratio & $r_c$ & 55 & $mgC.mgChla^{-1}$\\
water\ absorption\ in\ red & $k_{ro}$& 0.225 & $m^{-1}$\\
water\ absorption\ in\ green & $k_{go}$& 0.0232 & $m^{-1}$\\
pigment\ absorption\ in\ red & $k_{rp}$& 0.037 &
$m^{-1}.(mgChl.m^{-3})^{-l_r}$ \\
pigment\ absorption\ in\ green & $k_{gp}$& 0.074 & $m^{-1}.(mgChl.m^{-3})^{-l_g}$ \\
power law for absorption in red & $l_{r}$ & 0.629 & \\
power law for absorption in green & $l_{g}$ & 0.674 & \\
\hline
\end{tabular}
\end{table}

\vspace*{1pt}\textlineskip	%) USE THIS MEASUREMENT WHEN THERE IS
\section{Conclusion}
\vspace*{-0.5pt}
\noindent
We have presented a qualitative analysis of a one-dimensional biological 
NPZD-model. This model describes the evolution over time and space of 
four biological variables, 
phytoplankton, zooplankton, nutrients and detritus. 
The only physical process which is taken into account is vertical diffusion and 
the biological model is imbedded in a physical turbulence model 
which we did not give explicitly but appeared as a space and time-dependent 
mixing coefficient. 
The model's equation for detritus also contains an advection term which represents 
the sinking of detritus with a constant speed. All four variables 
interact through nonlinear reaction terms which depend on space and time through 
the action of light and present a discontinuity in the space variable at a particular 
depth.

We have formulated an initial-boundary value problem and proved 
existence of a unique weak
solution to it. Furthermore, a detailed investigation of the
reaction terms enabled us to prove positivity of the solution. This is
biologically important since variables represent concentrations which
should always be positive quantities. We have also shown that the result still holds 
if different parameterizations of biological processes found in 
the biogeochemical modeling literature are used.

The analysis conducted in this paper is a necessary first step towards the 
investigation of qualitative properties other than positivity which 
might be of interest. For example Boushaba \ea \cite{Boushaba:2002} 
deal with the problem of determining the asymptotic behavior of solutions 
to their phytoplankton model. One could also wish to investigate the bifurcational 
structure of the NPZD-model even though the complexity 
of the analytical formulation of the equations might constitute 
a difficulty. This type of study could help the understanding of the modifications 
of evolution of the NPZD system under minor changes in the values of parameters reported 
in Edwards \cite{Edwards:2001} and Faugeras \ea \cite{Faugeras:2002}. 

Eventually we would like to point out that the analysis we presented 
can easily be extended to models containing any
number, $n$, of biological variables as long as the nonlinearities allow us to
define a continuous nonlinear operator ${\bf G}$ satisfying
$||{\bf GC}||_{L^2(0,T,(L^2(0,L))^n)} \le M_g
||{\bf C}||_{L^2(0,T,(L^2(0,L))^n)}$. However, in such more complex models, 
the question of positivity seems to be delicate as 
equations have to be treated one after the
other, and the right order has to be found as the positivity
of some variables can depend on the positivity of others. 
Let us also mention that the analysis can be extended to three-dimensional models in which 
not only mixing coefficients but also velocities, calculated by an ocean circulation model, are included 
in the system of partial differential equations constituing the biological model. 
These velocities can be included in the bilinear form, $a(t,.,.)$, as $v_d$, the detritus sedimentation
speed, is in the formulation of the initial-boundary value problem we studied.\\   
~\\
{\bf Acknowledgements}.\\
The author is grateful to his thesis supervisors, Jacques Blum and Jacques Verron, 
for their guidance. The author also thanks Marina L\'evy and Laurent M\'emery for helpful 
discussions on physical and biological models, 
and Jean-Pierre Puel for his mathematical advices.

\end{document}

%% file: schemaNPZD_english.tex
\begin{picture}(0,0)%
\special{psfile=schemaNPZD_english.ps}%
\end{picture}%
\setlength{\unitlength}{1973sp}%
\begingroup\makeatletter\ifx\SetFigFont\undefined
% extract first six characters in \fmtname
\def\x#1#2#3#4#5#6#7\relax{\def\x{#1#2#3#4#5#6}}%
\expandafter\x\fmtname xxxxxx\relax \def\y{splain}%
\ifx\x\y   % LaTeX or SliTeX?
\gdef\SetFigFont#1#2#3{%
  \ifnum #1<17\tiny\else \ifnum #1<20\small\else
  \ifnum #1<24\normalsize\else \ifnum #1<29\large\else
  \ifnum #1<34\Large\else \ifnum #1<41\LARGE\else
     \huge\fi\fi\fi\fi\fi\fi
  \csname #3\endcsname}%
\else
\gdef\SetFigFont#1#2#3{\begingroup
  \count@#1\relax \ifnum 25<\count@\count@25\fi
  \def\x{\endgroup\@setsize\SetFigFont{#2pt}}%
  \expandafter\x
    \csname \romannumeral\the\count@ pt\expandafter\endcsname
    \csname @\romannumeral\the\count@ pt\endcsname
  \csname #3\endcsname}%
\fi
\fi\endgroup
\begin{picture}(11969,7694)(54,-7733)
\put(3526,-1486){\makebox(0,0)[lb]{\smash{\SetFigFont{6}{7.2}{rm}Z}}}
\put(8551,-1486){\makebox(0,0)[lb]{\smash{\SetFigFont{6}{7.2}{rm}P}}}
\put(3601,-4936){\makebox(0,0)[lb]{\smash{\SetFigFont{6}{7.2}{rm}D}}}
\put(8401,-4936){\makebox(0,0)[lb]{\smash{\SetFigFont{6}{7.2}{rm}   N}}}
\put(8701,-6286){\makebox(0,0)[lb]{\smash{\SetFigFont{6}{7.2}{rm}  winter}}}
\put(8701,-6511){\makebox(0,0)[lb]{\smash{\SetFigFont{6}{7.2}{rm}  convection}}}
\put(5326,-4861){\makebox(0,0)[lb]{\smash{\SetFigFont{6}{7.2}{rm}remineralization}}}
\put(8776,-2386){\makebox(0,0)[lb]{\smash{\SetFigFont{6}{7.2}{rm}production}}}
\put(5326,-1336){\makebox(0,0)[lb]{\smash{\SetFigFont{6}{7.2}{rm}grazing, $G_pZ$}}}
\put(2851,-2611){\makebox(0,0)[lb]{\smash{\SetFigFont{6}{7.2}{rm}$G_dZ$}}}
\put(1726,-4261){\makebox(0,0)[lb]{\smash{\SetFigFont{6}{7.2}{rm}mortality, $m_z Z$}}}
\put(1351,-3961){\makebox(0,0)[lb]{\smash{\SetFigFont{6}{7.2}{rm}$(1-a_p)G_pZ+(1-a_d)G_dZ$}}}
\put(1876,-3736){\makebox(0,0)[lb]{\smash{\SetFigFont{6}{7.2}{rm}fecal
pellets}}}
\put(2401,-2386){\makebox(0,0)[lb]{\smash{\SetFigFont{6}{7.2}{rm}grazing}}}
\put(5101,-3961){\makebox(0,0)[lb]{\smash{\SetFigFont{6}{7.2}{rm}   mortality}}}
\put(4951,-4261){\makebox(0,0)[lb]{\smash{\SetFigFont{6}{7.2}{rm}$m_p P$}}}
\put(6826,-3586){\makebox(0,0)[lb]{\smash{\SetFigFont{6}{7.2}{rm}$\mu_z Z$}}}
\put(7201,-3886){\makebox(0,0)[lb]{\smash{\SetFigFont{6}{7.2}{rm}excretion}}}
\put(8701,-3736){\makebox(0,0)[lb]{\smash{\SetFigFont{6}{7.2}{rm} exsudation}}}
\put(8701,-3961){\makebox(0,0)[lb]{\smash{\SetFigFont{6}{7.2}{rm}$\gamma \mu_p L_I L_N P$}}}
\put(2626,-5986){\makebox(0,0)[lb]{\smash{\SetFigFont{6}{7.2}{rm}$v_d \frac{\partial D}{\partial z}$}}}
\put(2176,-5686){\makebox(0,0)[lb]{\smash{\SetFigFont{6}{7.2}{rm}sedimentation}}}
\put(5626,-5161){\makebox(0,0)[lb]{\smash{\SetFigFont{6}{7.2}{rm}$\mu_d D$}}}
\put(8776,-2611){\makebox(0,0)[lb]{\smash{\SetFigFont{6}{7.2}{rm}$\mu_p L_IL_{N}P$}}}
\end{picture}